\documentclass[reqno]{amsart}
\usepackage{amssymb, latexsym, amsmath}
\usepackage{amstext}
\usepackage{amscd}
\usepackage{amsxtra}
\usepackage{graphicx}
\usepackage{epstopdf}
\usepackage[final]{hyperref}
\hypersetup{colorlinks=true, linkcolor=blue, anchorcolor=blue,
citecolor=red, filecolor=blue, menucolor=blue, pagecolor=blue,
urlcolor=blue}

\newcommand{\FT}{\,\, \widehat{} \,\,}
\newcommand{\init}{\vert_{t = 0}}

\newcommand{\abs}[1]{\left\vert #1 \right\vert}

\newcommand{\norm}[1]{\left\Vert #1 \right\Vert}

\newcommand{\bignorm}[1]{\bigl\Vert #1 \bigr\Vert}

\newcommand{\C}{\mathbb{C}}

\newcommand{\R}{\mathbb{R}}

\newcommand{\innerprod}[2]{\left\langle \, #1 , #2 \,
\right\rangle}

\newcommand{\angles}[1]{\langle #1 \rangle}

\DeclareMathOperator{\re}{Re} 
 
 \DeclareMathOperator{\supp}{supp}

\newtheorem{theorem}{Theorem}

\newtheorem{lemma}{Lemma}
\newtheorem{corollary}{Corollary}

\theoremstyle{definition}

\theoremstyle{remark}
\newtheorem{remark}{Remark}

\title[DKG in one space dimension]{Global
Well-posedness of the 1D Dirac-Klein-Gordon system in Sobolev spaces
of negative index}

\author[Achenef Tesfahun]{Achenef Tesfahun }

\thanks{Supported by the Research Council of Norway, project no.\
160192/V30, PDE and Harmonic Analysis. Address: Norwegian University
of Science and Technology, Department of Mathematical Sciences,
Alfred Getz' vei 1, N-7491 Trondheim, Norway. Email:
tesfahun@math.ntnu.no}


\begin{document}

\maketitle
\begin{abstract}
We prove that the Cauchy problem for the Dirac-Klein-Gordon system
of equations in 1D is globally well-posed in a range of Sobolev
spaces of negative index for the Dirac spinor and positive index for
the scalar field. The main ingredient in the proof is the theory of
``almost conservation law'' and ``$I$-method'' introduced by
Colliander, Keel, Staffilani, Takaoka and Tao. Our proof also relies
on the null structure in the system, and bilinear spacetime
estimates of Klainerman-Machedon type.
\end{abstract}

\section{Introduction} We consider the
Dirac-Klein-Gordon system (DKG) in one space dimension,
\begin{equation}\label{DKG1}
\left\{
\begin{aligned}
&  \left(-i(\gamma^0\partial_t +\gamma^1\partial_x)+M\right)\psi
=\phi \psi,
   \\
& (-\square  + m^2) \phi= \innerprod{\gamma^0 \psi}{\psi}_{\C^2},
\qquad ( \square = -\partial_t^2 +
\partial_x^2)
\end{aligned}
\right.
\end{equation}
with initial data
\begin{equation}\label{data0}
   \psi \init = \psi_0 \in H^s, \qquad \phi \init =
\phi_0 \in H^r, \qquad
\partial_t \phi \init = \phi_1 \in H^{r-1}.
\end{equation}
Here \ ($t,x)\in \R^{1+1}$, \ $\psi=\psi(t,x) \in \C^2$ is the Dirac
spinor and $\phi=\phi(t,x)$ is the scalar field which is
real-valued; $M, m>0$ are constants. Further, $
\innerprod{w}{z}_{\C^2}=z^*w$ for column vectors $w, z \in \C^2$,
where $z^*$ is the complex conjugate transpose of $z$; $
H^s=(1-{\partial_x^2})^{-s/2}L^2(\R)$ is the standard Sobolev space
of order $s$, and $\gamma^0$ and $\gamma^1$ are the Dirac matrices
given by
  $$ \gamma^0 =    \begin{pmatrix}
     0 & 1  \\
     1 &0
   \end{pmatrix},
   \qquad
   \gamma^1=  \begin{pmatrix}
     0 & -1  \\
     1 & 0
   \end{pmatrix}.
$$
We remark that with this choice the general requirements for Dirac
matrices are verified:
$$
\gamma^\mu\gamma^\nu+ \gamma^\nu\gamma^\mu=2g^{\mu\nu}I, \qquad
(\gamma^0)^{*}=\gamma^0, \qquad (\gamma^1)^{*}=-\gamma^1
$$
for $\mu, \nu=0,1$, where $(g^{\mu\nu})= ( \begin{smallmatrix}
     1 & 0  \\
     0 & -1
   \end{smallmatrix}$).

  We are interested in studying low regularity global solutions of the
  DKG system \eqref{DKG1} given the initial data \eqref{data0}.
   Global well-posedness (GWP) of DKG in 1d was first proved by
Chadam ~\cite{c1973} for data $$(\psi_0, \phi_0, \phi_1)\in
H^1\times H^1\times L^2. $$ This result has been improved over the
years in the sense that the regularity requirements on the initial
data which ensure global-in-time solutions can be lowered. The
earlier known GWP results for DKG in 1d are summarized in Table
\ref{Table1}.

\begin{table} \caption{GWP for DKG in 1d for data $(\psi_0, \phi_0,
\phi_1)\in H^s \times H^r \times H^{r-1}$.} \label{Table1}
\def\arraystretch{1.3}
\begin{center}
\begin{tabular}{|c|c|c|c|c|c|}
   \hline
     & $s$ & $r$
   \\
   \hline
   \hline
  Chadam  ~\cite{c1973}, 1973 & $1$ & $1$
   \\
   \hline
   Bournaveas ~\cite{b2000}, 2000 & $0$ & $1$
   \\
   \hline
Fang  ~\cite{f2001},  2001 & $0$ & $(1/2, 1]$
   \\
   \hline
Bournaveas and Gibbeson  ~\cite{bg2006}, 2006 & $0$ & $(1/4, 1]$
   \\
   \hline
Machihara ~\cite{m2006}, Pecher ~\cite{p2006}, 2006 & $0$ & $(0, 1]$
      \\
\hline
   Selberg  ~\cite{s2007}, 2007 & $(-1/8,0)$ & $(-s+\sqrt{s^2-s}, 1+s]$
  \\
\hline
\end{tabular}
\end{center}
\end{table}

\medskip\noindent

It is well known that when $s\ge 0$, the question of GWP of
\eqref{DKG1}, \eqref{data0} reduces to the corresponding local
question essentially due to the conservation of charge:
$$\norm{\psi(t,.)}_{L^{2}}=\norm{\psi_0}_{L^{2}}.$$
However, when $s<0$ there is no applicable conservation law. So even
if we have a local well-posedness (LWP) result for $s_0<s<0$ for
some $s_0$, it seems that we are stuck when trying to extend this to
a global-in-time solution.

The first breakthrough for resolving such problems came from
Bourgain ~\cite{b1998} who considered the cubic, defocusing
nonlinear Schr\"{o}dinger (NLS) equation in 2d, and proved GWP of
NLS below the (conserved) energy norm, i.e., below $H^1$. The idea
behind this method for a PDE is to split the rough initial data
(data whose regularity is below the conserved norm; say the $L^2$
norm from now on) into low and high frequency parts, using a Fourier
truncation operator. Consequently, one splits the PDE into two,
corresponding to the initial data with low and high frequencies. The
data with low frequency becomes smoother, in fact it is in $L^2 $ ,
so by global well-posedness its evolution remains in $L^2$ for all
time.

On the other hand, the difference between the original solution and
the evolution of the low frequency data satisfies a modified
nonlinear equation evolving the high-frequency part of the initial
data. The homogeneous part of this evolution is of course no
smoother than the initial data (so it may not be in $L^2$), but the
inhomogeneous part may be better due to nonlinear smoothing effects.
If the nonlinear smoothing brings the inhomogeneous part into $L^2$,
then at the end of the time interval of existence this part can be
added to the evolution of the low-frequency data, and the whole
process can be iterated. Assuming that sufficiently good a priori
estimates are available, this iteration allows one to reach an
arbitrarily large existence time, by adjusting the frequency cut-off
point of the original initial data. Several authors used Bourgain's
method to prove GWP of dispersive and wave equations with rough
data.

Recently, Selberg ~\cite{s2007} used Bourgain's method to prove GWP
of 1d-DKG below the charge norm, obtaining the following result (for
a comparison with earlier results, see Table \ref{Table1}):
\begin{theorem}\label{thm-boug-gwp} The DKG system \eqref{DKG1} is
   GWP for data \eqref{data0} provided
$$
-\frac18<s<0, \quad -s+\sqrt{s^2-s}< r\le 1+s.
$$
\end{theorem}

Concerning LWP of 1d-DKG the best result so far, which we state in
the next theorem, is due to S. Selberg and the present author
~\cite{st2007}, building on earlier results by several authors; see
~\cite{c1973}, ~\cite{b2000}, ~\cite{f2001}, ~\cite{bg2006},
~\cite{m2006} and ~\cite{p2006}.
\begin{theorem}\label{thm-lwp}
The DKG system \eqref{DKG1} is LWP for data \eqref{data0} if
$$
s>-\frac14, \quad r>0, \quad \abs{s}\le r\le 1+s.
$$
\end{theorem}
As mentioned earlier, when $s\ge 0$ this LWP result can be extended
to GWP result essentially due to the presence of conservation of
charge. So in view of Theorem \ref{thm-lwp}, we have the following
(see also Table \ref{Table1}):
\begin{theorem}\label{thm-consv-gwp}
The DKG system \eqref{DKG1} is GWP for data \eqref{data0} provided
$$
s\ge 0, \quad r>0, \quad \abs{s}\le r\le 1+s.
$$
\end{theorem}

However, in view of Theorems \ref{thm-lwp}, \ref{thm-consv-gwp} and
\ref{thm-boug-gwp}, there is still a gap left between the local and
global results known so far. In the present paper, we shall relax
the lower bound of $r$ in Theorem \ref{thm-boug-gwp}. In particular,
we fill the following gap left by Theorem ~\ref{thm-boug-gwp} (see
Figure \ref{fig:1}):
$$
-\frac18<s<0, \quad s+\sqrt{s^2-s}<r\le-s+\sqrt{s^2-s}
$$

 We now state our Main theorem.
\begin{theorem}\label{mainthm}
The DKG system \eqref{DKG1} is GWP for data \eqref{data0} if (see
Figure \ref{fig:1})
$$
-\frac18<s<0, \quad \quad s+\sqrt{s^2-s}< r\le 1+s.
$$
\end{theorem}

\begin{figure}[h]
   \centering
   \includegraphics[scale=.5]{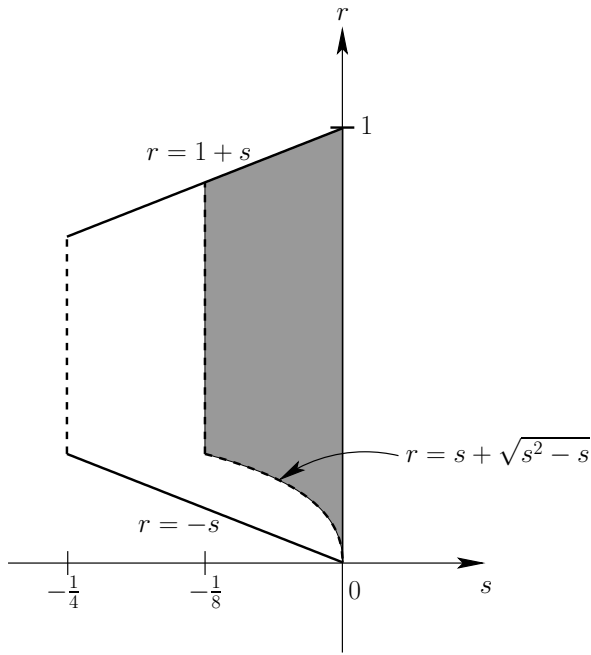}
   \caption{Global well-posedness of DKG holds in the interior of the shaded
region. Moreover, we can allow the line $r=1+s$ for $-1/8<s<0$. The
larger region which is contained in the strip $-1/4<s<0$ is where
Local well-posedness of DKG holds.}
   \label{fig:1}
\end{figure}

 The technique used here is the
theory of ``almost conservation law'' and ``$I$-method'' which was
developed by Colliander, Keel, Staffilani, Takaoka and Tao in a
series of papers; See for instance ~\cite{ckstt2001},
~\cite{ckstt2002}, ~\cite{ckstt2003}. The idea here is to apply a
smoothing operator $I$ to the solution of the PDE. The operator $I$
is chosen so that it is the identity for low frequencies and an
integration operator for high frequencies. The next step is to prove
an ``almost conservation law'' for the smoothed out solution as time
passes. Then one hopes that a modified version of LWP Theorem (after
$I$ is introduced) together with the ``almost conservation law''
will give a GWP result of the PDE for rough data.

In the DKG system, however, there is no conservation law for the
field $\phi$, only for the spinor $\psi$. Hence, we will not have
``almost conservation law'' for the $\phi$ field, which makes the
problem harder. To fix this problem we use a product estimate for
the Sobolev spaces for the inhomogeneous part of $\phi$, the
``almost conservation law'' for the spinor $\psi$, together with an
additional idea used by Selberg ~\cite{s2007} of making use of
induction argument involving a cascade of free waves.

 This paper is organized as follows. In the next section we fix
some notation, state definitions, and recall the derivation of the
conservation of charge. In Section 3 we shall state some basic
linear and bilinear estimates, and prove some null form estimates.
In Section 4 we discuss the $I$-method, state a modified LWP theorem
when we introduce the $I$ operator, state a key Lemma concerning
smoothing estimate,
  and show that
a combination of these imply  an ``almost conservation law'' for the
charge. Here, we also state another key Lemma which is used to
control the growth of solution of the Klein-Gordon part of DKG,
$\phi$. In Section 5 we put everything from section 4 together and
prove our main theorem. In Sections 6 and 7 we prove the two key
lemmas stated in section 4. In section 8 we prove the modified LWP
theorem.

\section{Preliminaries}

\subsection{Notation and Definitions}
In estimates, $C$ denotes a positive constant which can vary from
line to line and may depend on the Sobolev exponents $s$ and $r$ in
\eqref{data0}. We use the shorthand $X\lesssim Y $ for $X\le CY$,
and if $C\ll 1$ we use the symbol $\ll$ instead of $\lesssim$. We
 use the shorthand $X\approx Y$ for $Y\lesssim X\lesssim Y $.
 Throughout
the paper $\varepsilon$ is considered to be a sufficiently small
positive number in the sense that $0<\varepsilon\ll 1$. We also use
the notation
$$\angles{\cdot} = \sqrt{1 + \abs{\cdot}^2}.$$ The Fourier transforms
in space and space-time are defined by
$$
  \widehat f(\xi) = \int_{\R} e^{-ix\xi} f(x) \, dx,
  \qquad
  \widetilde u(\tau,\xi) = \int_{\R^{1+1}}
e^{-i(t\tau+x\xi)} u(t,x) \, dt \, dx.
$$
We denote $D = -i\partial_x$, so $\widehat{ D u}(\xi)= \xi \widehat
u(\xi)$. We also write $D_+:=\partial_t+\partial_x$ and
$D_-:=\partial_t-\partial_x$, hence $\square=-D_+D_-$.

We use the following spaces of Bourgain-Klainerman-Machedon type:
For $a, b \in \R$, define $X^{a,b}_\pm$, $H^{a,b}$ and $\mathcal
H^{a,b}$ to be the completions of $\mathcal S(\R^{1+1})$ with
respect to the norms
\begin{align*}
   \norm{u}_{X^{a,b}_\pm} &= \bignorm{\angles{\xi}^a
\angles{\tau\pm \xi}^b \widetilde u(\tau,\xi)}_{L^2_{\tau,\xi}},
   \\
   \norm{u}_{H^{a,b}} &= \bignorm{\angles{\xi}^a
\angles{\abs {\tau} - \abs{\xi}}^b \widetilde
u(\tau,\xi)}_{L^2_{\tau,\xi}},
  \\
  \norm{u}_{\mathcal H^{a,b}} &= \norm{u}_{H^{a,b}} +
\norm{\partial_t u}_{H^{a-1,b}}.
        \end{align*}
    We also need the
restrictions to a time slab $
  S_T = (0,T) \times \R$.  The
restriction $X_\pm^{a,b}(S_T)$ is a Banach space with norm
$$ \norm{u}_{X_\pm^{a,b}(S_T)} = \inf_{
\tilde{u}_{\mid {S_T}}=u } \norm{\tilde u}_{X_\pm^{a,b}}.$$ The
restrictions $H^{a,b}(S_T)$ and $\mathcal H^{a,b}(S_T)$ are defined
in the same way. See ~\cite{dfs2005} for more details about these
spaces.
\subsection{Rewriting DKG and Conservation of charge}
 To
see the symmetry in the DKG system, we shall rewrite \eqref{DKG1} as
follows:  Let $$\psi=
\begin{pmatrix}
     u \\
    v
   \end{pmatrix} $$ for $u, v\in \C$. Then we calculate
$$
(\gamma^0\partial_t +\gamma^1\partial_x)\psi=\begin{pmatrix}
     v_t-v_x \\
    u_t+u_x
   \end{pmatrix}
$$
and $$ \innerprod{\gamma^0 \psi}{\psi}_{\C^2}=\overline{u}
v+u\overline{v} =2\re(u\overline{v}).$$ Using this information, we
rewrite \eqref{DKG1} as
\begin{equation}\label{DKG2}
\left\{
\begin{aligned}
  & i(u_t +u_x)  =Mv - \phi v,\\
  & i(v_t -v_x)  =Mu - \phi u,
   \\
  &\square\phi=  m^2\phi -  2\re(u\overline{v}),
\end{aligned}
\right.
\end{equation}
with the initial data \eqref{data0} transformed to
\begin{equation}\label{data}
\left\{
\begin{aligned}
&u(0) = u_0 \in H^s, \qquad v(0)=v_0 \in H^s,\\
&\phi(0) = \phi_0 \in H^{r}, \qquad
\partial_t\phi(0)=\phi_1 \in
H^{r-1}.
\end{aligned}
\right.
\end{equation}
We shall then work with the Cauchy problem \eqref{DKG2},
\eqref{data} in the rest of the paper.

To motivate the derivation of the ``almost conservation law'', we
first recall the proof of the conservation of $L^2$-norm of the
solution to the Dirac part of the equation \eqref{DKG2}, using
integration by parts. To do this we first assume $u, v$ to be smooth
functions that decay at spatial infinity. For general well posed
solutions of \eqref{DKG2} where $s\ge 0$, the conservation of charge
will follow by a density argument.

Multiplying the first and second equations in \eqref{DKG2} by
$-i\overline u$ and $-i\overline v$, respectively, we get
\begin{equation*}
\left\{
\begin{aligned}
  & \overline u u_t + \overline u u_x  =-i M\overline
u v+ i \phi\overline
u v,\\
  & \overline v v_t -\overline v v_x  =-iMu\overline
v+i\phi u\overline v.
   \end{aligned}
\right.
\end{equation*}

Adding these two we obtain
$$
\overline u u_t + \overline v v_t+\overline u u_x-\overline v v_x
=2i(-M+\phi) \re(u\overline v).
$$
We now take the real part of this equation to get
$$
\re(\overline u u_t) + \re(\overline v v_t)+\re(\overline u
u_x)-\re(\overline v v_x )=0.
$$
Using the identity $( \overline u u)_t=\overline u u_t+\overline u_t
u= 2\re(\overline u u_t)$ (and the same identity if we take partial
derivative in $x$), we have
$$
(\abs{u}^2)_t+(\abs{v}^2)_t+(\abs{u}^2)_x-(\abs{v}^2)_x=0.
$$
We get after integrating in $x$
\begin{equation*}
\frac{d}{dt}\left(\norm{u(t)}_{L^2}^2+\norm{v(t)}_{L^2}^2\right)=0,
\end{equation*}
which implies the conservation charge:
\begin{equation}\label{conserv-charge}
\norm{u(t)}_{L^2}^2+\norm{v(t)}_{L^2}^2=\norm{u_0}_{L^2}^2+\norm{v_0}_{L^2}^2.
\end{equation}

\section{Linear and bilinear estimates}
The representation formula in Fourier space for the inhomogeneous
Dirac Cauchy problem
\begin{equation}\label{Cp:lin-Dirac}
\left\{
\begin{aligned}
&iD_\pm w_\pm = Mw_\pm+F_\pm(t,x) ,\\
& w_\pm(0,x)=f_\pm(x),
\end{aligned}
\right.
\end{equation}
is given by
\begin{equation}\label{sol:lin-Dirac}
\widehat{ w_\pm(t)}(\xi)=e^{-i(M\pm\xi)t}\hat f_\pm(\xi)+\int_0^t
e^{-i(M\pm \xi)(t-t')}\widehat F_\pm(t',\xi)dt'.
\end{equation}
 Similarly, the representation formula in Fourier space  for the inhomogeneous
Klein-Gordon Cauchy problem
\begin{equation}\label{Cp:lin-KG}
\left\{
\begin{aligned}
&\square z = m^2z+F(t,x),\\
& z(0,x)=f(x),\quad  \partial_tz(0,x)=g(x),
\end{aligned}
\right.
\end{equation}
is given by
\begin{equation}\label{sol:lin-KG}
\widehat{z(t)}(\xi)=\cos(t\angles{\xi}_m)\widehat{f}(\xi)+\frac{\sin(t\angles{\xi}_m)}{\angles{\xi}_m}\widehat{g}(\xi)
+\int_0^t
\frac{\sin\left((t-t')\angles{\xi}_m\right)}{\angles{\xi}_m}\widehat{F(t')}(\xi)dt',
\end{equation}
 where $\angles{\xi}_m = \sqrt{m^2 + \abs{\xi}^2}$.
\subsection{Linear estimates}
Throughout the paper, we use the notation
$$
\norm{ z[ t] }_{H^a}\equiv \norm{ z( t) }_{H^{a}}+ \norm{
\partial_t z(t) }_{H^{a-1}}.
$$
From the solution formulas \eqref{sol:lin-Dirac} and
\eqref{sol:lin-KG} we deduce the following energy estimates for the
solution of Cauchy problems \eqref{Cp:lin-Dirac} and
\eqref{Cp:lin-KG}, respectively:
\begin{align}
\label{Est:lin-Dirac} &\norm{w_\pm(t)}_{H^a}\le\norm{f_\pm}_{H^a}+
\int_0^t
\norm{F_\pm(t')}_{H^a}dt',\\
 \label{Est:lin-KG}
&\norm{z[t]}_{H^a}\le C\left(\norm{f}_{H^a}+
\norm{g}_{H^{a-1}}+\int_0^t \norm{F(t')}_{H^{a-1}}dt'\right),
\end{align}
\footnote[1]{If we set $m=0$ in \eqref{Cp:lin-KG}, then the constant
$C$ in the energy estimate \eqref{Est:lin-KG} will depend on $t$.}
for some $C>0$ and for all $t>0$.

 The estimates we present in the following two lemmas
 are a priori estimates for the solutions of
 the massive Dirac and Klein-Gordon Cauchy problems, and
  they are crucial for the reduction of the
local existence problem to bilinear estimates. These estimates are
 variants of the estimates in \cite[Lemma~5, Lemma~6]{dfs2005}, i.e., when $M=m=0$.
 \begin{lemma}\label{DEElemma}
Let $1/2<b\le 1$, $a\in \R$, $0<T\le 1$ and  $0\le \delta\le 1-b$.
Then for all data $F_\pm\in {X^{a,b-1+ \delta}_\pm}(S_T)$ and
$f_\pm\in H^a$, we have the following estimate for the solution
\eqref{sol:lin-Dirac} of the Dirac Cauchy problem
\eqref{Cp:lin-Dirac}:
\begin{equation}\label{Apriori-Dirac}
\norm{w_\pm}_{X^{a,b}_\pm(S_T) }\le C \Bigl( \norm{f_\pm}_{H^a}+ T^
\delta \norm{F_\pm}_{X^{a,b-1+ \delta}_\pm(S_T)}\Bigr),
\end{equation}
where $C$ depends only on $b$.
\end{lemma}
\begin{proof}
Define the space $X_{M,\pm}^{s,\theta}$ with a norm
$$
\norm{u}_{X^{s,\theta}_{M,\pm}} = \bignorm{\angles{\xi}^s
\angles{\tau\pm \xi_M}^\theta \widetilde
u(\tau,\xi)}_{L^2_{\tau,\xi}},
$$
where $\xi_M=\xi+M$. In view of \cite[Lemma~5]{dfs2005} the estimate
\eqref{Apriori-Dirac} holds if we replace the space
$X_\pm^{s,\theta}$ by $X_{M,\pm}^{s,\theta}$. So, to complete the
proof it suffices to show
$$X_\pm^{s,\theta} =X_{M,\pm}^{s,\theta} .$$ This reduces to proving
\begin{equation}\label{X-XM}
\angles{\tau\pm \xi}\approx\angles{\tau\pm \xi_M}.
\end{equation}
But this follows from
$$
\angles{\tau\pm \xi} \approx 1+ \abs{\tau\pm \xi}\le 1+ \abs{\tau\pm
\xi_M }+M\lesssim\angles{\tau\pm\xi_M },
$$
and conversely, $$ \angles{\tau\pm\xi_M }\approx 1+ \abs{\tau\pm
\xi_M }\le 1+ \abs{\tau\pm \xi}+M\lesssim\angles{\tau\pm\xi}.
$$
\end{proof}

\begin{lemma}\label{KGEElemma}
Let $1/2<b\le 1$, $a\in \R$, $0<T\le 1$ and  $0\le \delta \le 1-b$.
Then for all data $F\in {H^{a-1,b-1+\delta}}(S_T)$, $f\in H^a$ and
$g \in H^{a-1}$, we have the following estimate for the solution
\eqref{sol:lin-KG} of the Klein-Gordon Cauchy problem
\eqref{Cp:lin-KG}: \begin{equation}\label{Apriori-KG}
\norm{z}_{\mathcal{H}^{a,b}(S_T)}\le C \Bigl( \norm{f}_{H^a}+
\norm{g}_{H^{a-1}}+ T^{ \delta/2}\norm{F}_{H^{a-1,b-1+
\delta}(S_T)}\Bigr), \end{equation} where $C$ depends only on $b$.
\end{lemma}
\begin{proof}
Define the space $H_m^{s,\theta}$ with a norm
$$
\norm{u}_{H^{s,\theta}_m} = \bignorm{\angles{\xi}^s \angles{\abs
{\tau} - \angles{\xi}_m}^\theta \widetilde
u(\tau,\xi)}_{L^2_{\tau,\xi}},
$$
and the space $\mathcal{H}_m^{s,\theta}$ with a norm
$$
\norm{u}_{\mathcal {H}_m^{s,\theta}} = \norm{u}_{H_m^{s,\theta}} +
\norm{\partial_t u}_{H_m^{s-1,\theta}}.
$$
So, in view of \cite[Theorem~12]{s99} the estimate
\eqref{Apriori-KG} holds if we replace the spaces $H^{s,\theta}$ and
$\mathcal{H}^{s,\theta}$ by $H_m^{s,\theta}$ and
$\mathcal{H}_m^{s,\theta}$, respectively. Hence, to complete the
proof it suffices to show
$$H^{s,\theta} =H_m^{s,\theta} .$$ This reduces to proving
\begin{equation}\label{H-Hm}
\angles{\abs {\tau} - \abs{\xi}}\approx\angles{\abs {\tau} -
\angles{\xi}_m}.
\end{equation}
Assume $\tau\ge 0$. Then
\begin{align*}
\angles{-\tau+\abs{\xi}} \approx 1+ \abs{-\tau+\abs{\xi}}&\le 1+
\abs{-\tau+\angles{\xi}_m }+ \angles{\xi}_m-\abs{\xi}\\
&\le 1+m+\abs{-\tau+\angles{\xi}_m }\\
&\lesssim\angles{-\tau+\angles{\xi}_m }.
\end{align*}
Conversely,
\begin{align*} \angles{ -\tau+\angles{\xi}_m }&\approx 1+
\abs{-\tau+\angles{\xi}_m }\\&\le 1+
\abs{-\tau+\abs{\xi}}+\angles{\xi}_m-\abs{\xi} \\
&=1+m+\abs{-\tau+\abs{\xi}}\lesssim\angles{-\tau+\abs{\xi}}.
\end{align*}
Similarly, it can be shown that the estimate \eqref{H-Hm} holds true
for $\tau< 0$. This completes the proof of the Theorem.
\end{proof}
We shall need the fact that if $ b>1/2$, then
\begin{align}\label{embxhc}
\norm{u(t)}_{H^a}\le C\norm{u}_{H^{a,b}(S_T)}\le
C\norm{u}_{X_\pm^{a,b}(S_T)} \qquad \text{for} \  0\le t\le T,
\end{align}
where $C$ depends only on $b$. The following estimate will also be
needed in the last section (see ~\cite{p2005} for the proof):
\begin{align}\label{Tfactor}
\norm{u}_{X_\pm^{a,\varepsilon}(S_T)}\le
CT^{1/2-2\varepsilon}\norm{u}_{X_\pm^{a,1/2-\varepsilon}(S_T)},
\end{align}
for all $\varepsilon>0$ sufficiently small, and $0<T\le 1$.

\begin{lemma}\label{Sob-emb-t} Let $2\le q\le\infty$ and $\varepsilon>0$ be sufficiently small. Then
$$
 \norm{u}_{H^{0,-1/2+1/q-\varepsilon}} \lesssim \norm{u}_{L_t^{q'} L_x^2}
$$
where $\frac1q+\frac{1}{q'}=1$.
\end{lemma}

\begin{remark} This Lemma also holds if we replace $H^{0,-1/2+1/q-\varepsilon}$
by $X_\pm^{0,-1/2+1/q-\varepsilon}$, simply because
$H^{0,\alpha}\hookrightarrow X_\pm^{0,\alpha}$ for any $\alpha\le
0$.
\end{remark}
\begin{proof}[Proof of lemma \ref{Sob-emb-t}] By duality, the estimate is equivalent to
\begin{equation}\label{sobin-t}
  \norm{u}_{L_t^q L_x^2} \lesssim
\norm{u}_{H^{0,  1/2-1/q+\varepsilon}}.
\end{equation}
 By Sobolev embedding in $t$
$$ \norm{u}_{L_t^\infty L_x^2} \lesssim
\norm{u}_{H^{0,  1/2+\varepsilon}}.$$
 Interpolating this with
$$ \norm{u}_{L_t^2 L_x^2} =
\norm{u}_{L_t^2 L_x^2}$$ gives
$$ \norm{u}_{L_t^q L_x^2} \lesssim
\norm{u}_{H^{0,  b}}$$ where
$$\frac{1}{q}=\frac{\theta}{\infty}+\frac{1-\theta}{2}, \qquad
b= \theta(1/2+\varepsilon) $$ for $\theta\in [0,1]$. Thus
$b=\frac{1}{2}-\frac{1}{q}+\varepsilon(1-\frac{2}{q})<\frac{1}{2}-\frac{1}{q}+\varepsilon$,
and hence \eqref{sobin-t} follows. This concludes the proof of the
Lemma.
\end{proof}

\subsection{Bilinear estimates}
 We
shall need the standard product estimate for the Sobolev spaces $H^s
$, which reads as follows:
\begin{lemma}\label{Sob-emb-thm}
Suppose $a_1,a_2,a_3 \in \R$. Then
\begin{equation}\label{Sob-prodembed:1}
\norm{fg}_{H^{-a_3}}\lesssim \norm{f}_{ H^{a_1} }\norm{g}_{
H^{a_2}}.
  \end{equation}
provided
\begin{equation}\label{prodcond}
\begin{aligned}
      &a_1+a_2+a_3 > 1/2,\\
&  a_1+a_2 \ge 0, \quad  a_1+a_3 \ge 0, \quad a_2+a_3\ge 0.
\end{aligned}
\end{equation}
\end{lemma}

The following estimate is just the analogue of Lemma
\ref{Sob-emb-thm} for the wave-Sobolev space $H^{s,b}$.
\begin{lemma}~\cite{s99,st2007}\label{Thmembedding}.
Suppose $a_1,a_2,a_3 \in \R$ satisfy \eqref{prodcond}. Let $\alpha,
\beta, \gamma\ge 0$ and $\alpha+\beta+\gamma>\frac{1}{2}$. Then
\begin{equation}\label{prodembed:1}
\norm{wz}_{H^{-a_3,-\gamma}}\lesssim \norm{w}_{ H^{a_1, \alpha}
}\norm{z}_{ H^{a_2, \beta}}.
  \end{equation}
\end{lemma}

   The following comparison estimate between elliptic and hyperbolic
   weights proved in ~\cite{st2007} will be needed in the proof of
   Lemma \ref{lemmambedding:1} below. This estimate
is used to identify null structure in bilinear estimates.
\begin{lemma}\label{lemma-comparsion}
Denote
$$ \Gamma = \abs{\tau}-\abs{\xi},
  \qquad \Theta_+ = \lambda+\eta,
  \qquad \Sigma_- =\tau-\lambda-(\xi-\eta). $$
  Then
  $$\min(\abs{\eta},\abs{\xi-\eta})\lesssim
\max(\abs{\Gamma},\abs{\Theta_+},\abs{\Sigma_-} ).$$
\end{lemma}
We now prove the following null form estimates. We remark that the
null structure of DKG in 1d is reflected in the difference of signs
in the r.h.s. of the estimate \eqref{reduc-embedding:1}, and the
difference of signs in the r.h.s. and l.h.s. of the estimates
\eqref{prodembed:-+} and \eqref{prodembed:+-} below; for equal signs
the estimates would fail.
\begin{lemma}\label{lemmambedding:1}
Let $b=\frac12+ \varepsilon$ for sufficiently small $\varepsilon>0$.
 The bilinear estimates
\begin{alignat}{2}
\label{reduc-embedding:1}
  \norm{wz}_{H^{-s_1,b-1}}\lesssim
\norm{w}_{ X_+^{s_2, b} }\norm{z}_{ X_-^{s_3,
 b}},
  \\
\label{prodembed:-+} \norm{wz}_{X_-^{-s_3,b-1}}\lesssim \norm{w}_{
H^{s_1,b} }\norm{z}_{ X_+^{s_2,
 b}},\\
\label{prodembed:+-} \norm{wz}_{X_+^{-s_3,b-1}}\lesssim \norm{w}_{
H^{s_1, b} }\norm{z}_{ X_-^{s_2,
  b}}
  \end{alignat}
  hold
  provided
  \begin{equation}\label{cond:embed1}
  \begin{split}
  &s_1+s_2+s_3> \varepsilon,\\
  &s_2+s_3\ge -1/2+\varepsilon,\\
& s_1+s_2\ge 0, \quad s_1+s_3\ge 0.
  \end{split}
  \end{equation}
\end{lemma}
\begin{remark}\label{remark-prodest}
 The bilinear estimates
\eqref{reduc-embedding:1}--\eqref{prodembed:+-} will still hold if
we replace $z$ in the l.h.s. of the inequalities in these estimates
by $\overline z$. We also note that these bilinear estimates will
imply the corresponding estimates where the spaces are restricted in
time (refer ~\cite{dfs2005} for the detail).
\end{remark}

\begin{proof}[Proof of Lemma \ref{lemmambedding:1}]
We only prove \eqref{reduc-embedding:1} and \eqref{prodembed:-+},
since \eqref{prodembed:+-} will follow from \eqref{prodembed:-+} by
symmetry. We first prove \eqref{reduc-embedding:1}.

Set
\begin{align*}
& F_+(\lambda,\eta) = \angles{\eta}^{s_2}\angles{\lambda+\eta}^{b}
\abs{\widetilde w(\lambda,\eta)},
  \\
& G_-(\lambda,\eta) =\angles{\eta}^{s_3} \angles{\lambda-\eta}^{b}
  \abs{\widetilde z(\lambda,\eta)}.
\end{align*}
Then \eqref{reduc-embedding:1} is equivalent to
$$J\lesssim \norm{F_+}_{L^2}\norm{G_-}_{L^2}$$
where
$$J:=\norm{\int_{\R^{1+1} }
  \frac{
F_+(\lambda,\eta) G_-(\tau-\lambda,\xi-\eta)\, d\lambda \, d\eta  }
  { {\angles{\xi}^{s_1} \angles{\eta}^{s_2}
\angles{\xi-\eta}^{s_3}\angles{\Gamma}^{1-b}
   \angles{\Theta_+}^{b}
\angles{\Sigma_-}^{b}} }}_{L^2_{\tau,\xi}},
$$
where $\Gamma$, $\Theta_+$ and $ \Sigma_-$ are defined as in Lemma
\ref{lemma-comparsion}.

 By symmetry, we may assume $\abs{\eta} \le \abs{\xi-\eta}$. If
$\max(\abs{\Gamma},\abs{\Theta_+},\abs{\Sigma_-} )=\abs{\Gamma}$,
then in view of Lemma \ref{lemma-comparsion} the estimate for $J $
reduces to \eqref{prodembed:1} with exponents
$(a_1,a_2,a_3)=(s_2+1-b, s_3, s_1)$, \ $(\alpha,\beta,\gamma)=(b, b,
0)$. If $\max(\abs{\Gamma},\abs{\Theta_+},\abs{\Sigma_-}
)=\abs{\Theta_+}$ or $\abs{\Sigma_-}$, then the estimate for $J $
reduces to \eqref{prodembed:1} with exponents $(a_1,a_2,a_3)=(s_2+b,
s_3, s_1)$, \ $(\alpha,\beta,\gamma)=(0, b, 1-b)$ or $(b, 0, 1-b)$.

Then the conditions on $(a_1, a_2, a_3$), \eqref{prodcond}, will be
satisfied (for all the cases above) as long as \eqref{cond:embed1}
holds.

Next, we prove \eqref{prodembed:-+}. By duality, proving the
estimate \eqref{prodembed:-+} is equivalent to proving
\begin{equation}\label{prod-duality}
  \norm{wz}_{H^{-s_1,-b}}\lesssim
\norm{w}_{ X_+^{s_2, b} }\norm{z}_{ X_-^{s_3,
  1-b}},
\end{equation}
where $w, \ z$ are $\C$-valued functions. Define $F_+$ as before,
and redefine
 $G_-$ as
$$
G_-(\lambda,\eta)=\angles{\eta}^{s_3} \angles{\lambda-\eta}^{1-b}
  \abs{\widetilde z(\lambda,\eta)}.
$$ Then \eqref{prod-duality} is equivalent to
$$L\lesssim \norm{F_+}_{L^2}\norm{G_-}_{L^2}$$
where
$$L:=\norm{\int_{\R^{1+1} }
  \frac{
F_+(\lambda,\eta) G_-(\tau-\lambda,\xi-\eta)\, d\lambda \, d\eta  }
  { {\angles{\xi}^{s_1} \angles{\eta}^{s_2}
\angles{\xi-\eta}^{s_3}\angles{\Gamma}^{b}
   \angles{\Theta_+}^{b}
\angles{\Sigma_-}^{1-b}} }}_{L^2_{\tau,\xi}}.
$$
We use the same argument as in the estimate for $J$ above. In view
of Lemma \ref{lemma-comparsion} we can add $1-b$ to the exponent of
either the weight $\angles{\eta}$ or $\angles{\xi-\eta}$, at the
cost of giving up one of the hyperbolic weights $\angles{\Gamma}$,
$\angles{\Theta_+}$ or $\angles{\Sigma_-}$. Then we apply Lemma
\ref{Thmembedding}. In fact, we can reduce the estimate for $L$ to
\eqref{prodembed:1} with exponents $(a_1,a_2,a_3)=(s_2+1-b, s_3,
s_1)$ or $(s_2, s_3+1-b, s_1)$. Then the condition \eqref{prodcond}
is satisfied, since we assume \eqref{cond:embed1}.
\end{proof}

\section{I-Method and Almost Conservation Law}

Let $s< 0$ and $N \gg 1$ be fixed. Define the Fourier multiplier
operator

\begin{equation}
\label{Idefined} {{\widehat{If}}}( \xi ) = q( \xi ) {\widehat{f}} (
\xi ), \qquad q( \xi ) = \left\{ \begin{matrix}
1, & |\xi | < N, \\
N^{-s} {{|\xi |}^s} , & |\xi | >2N,
\end{matrix}
\right.
\end{equation}
with $q$ even, smooth and monotone.

Observe that on low frequencies $\{ \xi : |\xi | <N \},~I$ is the
identity operator. The operator $I$ commutes with differential
operators. We also have the following properties: For $a, b\in \R$,
\begin{align}
\label{Ibdd-1} &\norm{If}_{H^a}\lesssim\norm{f}_{H^a}, \quad \norm{Iw}_{H^{a,b}}\lesssim\norm{w}_{H^{a,b}},\\
\label{comparision:Iu,u}
&\norm{f}_{H^s}\lesssim\norm{If}_{L^2}\lesssim N^{-s}\norm{f}_{H^s},
\\
 \label{I2bdd}
&\norm{f}_{H^a}\lesssim \norm{I^2f}_{H^{a-2s}} \lesssim
N^{-2s}\norm{f}_{H^a},
\end{align}
and
 if $\supp \hat z(t,\cdot) \subset \{\xi: \abs{\xi}\gtrsim N\}$, we
 have $$\norm{I^{-1}z}_{H^{a, b}}\lesssim
N^s\norm{z}_{H^{a-s, b}},$$ which in turn implies
\begin{equation}
\label{compare-Iz:z} \norm{Iz}_{H^{a, b}}\lesssim
N^s\norm{I^2z}_{H^{a-s, b}}.
\end{equation}

 Let $(s, r)$ be such that $ - \frac16 < s < 0$ and $-s\le
r<\frac12+2s$. Then from the modified LWP theorem which we state in
the next section, there exists a $\Delta T>0$ depending on
$$
\norm{Iu_0}_{L^2}+ \norm{Iv_0}_{L^2}+ \norm{I^2\phi_0}_{H^{r-2s}}+
\norm{I^2\phi_1}_{H^{r-2s-1}}, $$ such that \eqref{DKG2},
\eqref{data} has solution for times  $0\le t\le \Delta T$. Of
course, \eqref{DKG2}, \eqref{data} has solution  for $(s,r)$ in a
larger region as in Theorem \ref{thm-lwp}, but now we reprove the
Theorem in the above restricted region with a different time of
existence of solution.

 Now, we observe using the Fundamental Theorem of Calculus
that
\begin{align*}
{{\| Iu ( \Delta T) \|}_{L^2}^2 } + {{\| Iv ( \Delta T) \|}_{L^2}^2
}&= {{\| Iu_0 \|}_{L^2}^2 } + {{\| Iv_0 \|}_{L^2}^2 } +R_1(\Delta
T)+ R_2(\Delta T),
\end{align*}
where
\begin{align*} &R_1(\Delta T)=\int_0^{\Delta T} \frac{d}{d\tau } ( Iu(
\tau) , Iu ( \tau
) ) d\tau, \\
&R_2(\Delta T)=\int_0^{\Delta T} \frac{d}{d\tau } ( Iv( \tau) , Iv (
\tau ) ) d\tau,
\end{align*}
and $(.,.)$ denotes the scalar product in $L^2$. By the first
equation in \eqref{DKG2},
\begin{align*} R_1(T)&=\int_0^{\Delta T} \frac{d}{d\tau } ( Iu(
\tau) ,
Iu ( \tau ) ) d\tau\\
&=2\re\int_0^{\Delta T}  ( I\dot u( \tau) , Iu ( \tau ) )
d\tau\\
&=2\re\int_0^{\Delta T}  ( I \left[-iMv+i\phi v-u_x\right](\tau) ,
Iu ( \tau )
) d\tau\\
&=2\re\int_0^{\Delta T}  ( -iMIv(\tau), Iu ( \tau ) )
d\tau+2\re\int_0^{\Delta T} (iI (\phi v)(\tau) , Iu ( \tau ) )
d\tau\\
 &\quad +2\re\int_0^{\Delta T}  ( -I u_x(\tau) , Iu ( \tau ) )
d\tau.
\end{align*}
But the third term is zero. Indeed,
\begin{align*} 2\re\int_0^{\Delta T}  ( -I u_x(\tau) ,
Iu ( \tau ) ) d\tau&=-2\re\int_0^{\Delta T}\int_\R   \overline{I
u_x}(\tau) Iu ( \tau )  dx d\tau\\
&=-\int_0^{\Delta T}\int_\R \left(\overline{I u}(\tau) Iu ( \tau )
\right)_x dx d\tau=0.
\end{align*}
Hence
\begin{align*}
R_1(\Delta T)&= 2\re\int_0^{\Delta T}\int_\R  -iMIv(\tau)\overline{
Iu ( \tau )} dx d\tau+2\re\int_0^{\Delta T}\int_\R iI (\phi v)(\tau)
\overline{Iu ( \tau )}dx d\tau.
\end{align*}
Similarly, by the second equation in \eqref{DKG2}
\begin{align*}
R_2(\Delta T)&=2\re\int_0^{\Delta T}\int_\R  -iMIu(\tau)\overline{
Iv ( \tau )} dx d\tau+2\re\int_0^{\Delta T}\int_\R iI (\phi u)(\tau)
\overline{Iv ( \tau )}dx d\tau.
\end{align*}
We therefore get
\begin{align*}
R(\Delta T):&= R_1(\Delta T)+R_2(\Delta T)\\&=2\re\int_0^{\Delta
T}\int_\R -2Mi\re(Iu(\tau)\overline{ Iv ( \tau )} ) dx
d\tau\\
&\quad+2\re\int_0^{\Delta T}\int_\R iI (\phi u)(\tau) \overline{Iv (
\tau )}dx d\tau+2\re\int_0^{\Delta T}\int_\R iI (\phi v)(\tau)
\overline{Iu ( \tau )}dx
d\tau\\
&=2\re\int_0^{\Delta T}\int_\R iI (\phi u)(\tau) \overline{Iv ( \tau
)}dx d\tau+2\re\int_0^{\Delta T}\int_\R iI (\phi v)(\tau)
\overline{Iu ( \tau )}dx d\tau.
\end{align*}
Now, observe that $$ -iI\phi Iu \overline{Iv}-iI\phi Iv
\overline{Iu}=-2iI\phi \re( Iu \overline{Iv}).$$ Using this identity
and the fact that $I\phi$ is real-valued (recall that the multiplier
$q$ is assumed to be even), we obtain
$$
2\re\left[\int_0^{\Delta T}\int_\R-iI\phi(\tau)
Iu(\tau)\overline{Iv}(\tau)+\int_0^{\Delta T}\int_\R-iI\phi(\tau)
Iv(\tau) \overline{Iu}(\tau)dx d\tau\right]=0.
$$

We can therefore add this term to $R(\Delta T)$ for free. We remark
that adding this term to $R(\Delta T)$ gives us a cancellation on
the dangerous interaction in frequencies, and this makes it possible
for proving some smoothing estimates. This in turn enables us to get
the desired \emph{almost conservation law} (see below for the
details). We can now write
\begin{align*}
R(\Delta T)&= 2\re\int_0^{\Delta T}\int_\R i\{I (\phi u)-I\phi
Iu\}(\tau) \overline{Iv ( \tau )}dx d\tau\\
&\quad +2\re\int_0^{\Delta T}\int_\R i\{I (\phi v)-I\phi Iv\}(\tau)
\overline{Iu ( \tau )}dx d\tau.\end{align*} We therefore conclude
\begin{equation}\label{conser-IuIv}
\norm{ Iu (\Delta T) }_{L^2}^2  + \norm{Iv (\Delta T) }_{L^2}^2 =
\norm{ Iu_0 }_{L^2}^2  + \norm{ Iv_0 }_{L^2}^2  + R(\Delta T).
\end{equation}

The quantity that could make $\norm{ Iu ( \Delta T) }_{L^2}^2 +
\norm{Iv ( \Delta T) }_{L^2}^2 $ too large in the future is
$R(\Delta T)$. The idea is then to use bilinear estimates to show
that locally in time $R(\Delta T)$ is small. By Plancherel and
Cauchy-Schwarz, we obtain
\begin{equation}\label{estimate:R}
\begin{split}
\abs{R(\Delta T)}&\lesssim \norm{I (\phi u)-I\phi Iu}_{X_-^{0,
-b}(S_{\Delta T})}\norm{ Iv}_{X_-^{0,
b}(S_{\Delta T})}\\
&+\norm{I (\phi v)-I\phi Iv}_{X_+^{0, -b}(S_{\Delta T})}\norm{
Iu}_{X_+^{0,b}(S_{\Delta T})},
\end{split}
\end{equation}
for $b \in \R$.

 We denote
$$
Q_I(f,g)=I (f g)-If \cdot Ig.
$$

\begin{lemma}{(Smoothing estimate).}\label{lemma-errorestimate} Suppose
\begin{equation}\label{cond-Qest}
-1/3< s<0, \qquad -s<r\le 1+2s.
\end{equation}
 Let $b=\frac12+
\varepsilon$ for sufficiently small $\varepsilon>0$ depending on
$s,r$. Then
\begin{align}
\label{Q:phi-u} &\norm{Q_I(\phi,u)}_{X_-^{0, -b}(S_{\Delta T})}\le
CN^{-r+2s +2\varepsilon} \norm{I^2 \phi}_{H^{r-2s, b}(S_{\Delta
T})}\norm{ Iu}_{X_+^{0,
b}(S_{\Delta T})},\\
\label{Q:phi-v} &\norm{Q_I(\phi,v)}_{X_+^{0, -b}(S_{\Delta T})}\le
CN^{-r+2s+2\varepsilon}\norm{ I^2\phi}_{H^{r-2s, b}(S_{\Delta T})}
  \norm{ Iv}_{X_-^{0, b}(S_{\Delta T})},
     \end{align}
where $C$ depends on $s$, $r$, $\varepsilon$, but not $N$ or $\Delta
T$ .
\end{lemma}
In order to apply the $I$-method, we need a variant of Theorem
\ref{thm-lwp}, which we call a modified LWP Theorem for the
$I$-modified equation
\begin{equation}\label{IDKG2}
\left\{
\begin{aligned}
  &iD_+(Iu)=MIu-I(\phi v),\\
  & iD_-(Iv) =MIv-I(\phi u),
   \\
  &\square(I^2\phi)=m^2I^2\phi -  2I^2(\re(u \overline v)),
\end{aligned}
\right.
\end{equation}
which is obtained from \eqref{DKG2} by applying $I$. The
corresponding $I$-initial data obtained from \eqref{data} are
\begin{equation}\label{Idata}
\left\{
\begin{aligned}
&Iu(0) = Iu_0 \in L^2, \qquad Iv(0)=Iv_0 \in L^2,\\
&I^2\phi(0) = I^2\phi_0 \in H^{r-2s}, \qquad \partial_t
I^2\phi(0)=I^2\phi_1 \in H^{r-2s-1}.
\end{aligned}
\right.
\end{equation}

Combining \eqref{conser-IuIv}, \eqref{estimate:R}, \eqref{Q:phi-u}
and \eqref{Q:phi-v} we obtain, for $s,r $  and  $\varepsilon$ as in
Lemma \ref{lemma-errorestimate} ,
\begin{equation}\label{almostconser-IuIv}
\begin{split}
&\norm{ Iu ( \Delta T) }_{L^2}^2  + \norm{Iv ( \Delta T) }_{L^2}^2\\
&\le \norm{ Iu_0 }_{L^2}^2  + \norm{ Iv_0 }_{L^2}^2\\
&\quad + CN^{-r+2s+2\varepsilon}\norm{I^2\phi}_{H^{r-2s,
b}(S_{\Delta T})}\norm{Iu}_{X_+^{0, b}(S_{\Delta T})}
\norm{Iv}_{X_-^{0,b}(S_{\Delta T})},
\end{split}
\end{equation}
where $C$ depends on $s$, $r$ and $\varepsilon$, but not $N$ or
$\Delta T$.

In view of \eqref{comparision:Iu,u} and \eqref{I2bdd}, we have
\begin{equation}\label{u0v0p0}
\begin{split}
 &\norm{ Iu_0 }_{L^2}+
\norm{ Iv_0 }_{L^2}\le AN^{-s}, \\
& \norm{ I^2\phi_0 }_{H^{r-2s}}+ \norm{I^2\phi_1}_{H^{r-2s-1}}\le BN^{-2s}, \\
\end{split}
\end{equation}
for some $A, B>0$. Here, $A$ depends on $ \norm{ u_0 }_{L^2}+ \norm{
v_0 }_{L^2}$ whereas $B$ depends on $\norm{ \phi_0 }_{H^r}+
\norm{\phi_1}_{H^{r-1}}$.

We now state the modified LWP theorem which will be proved in the
last section.
\begin{theorem}\label{theorem-modifiedlwp}
Suppose
\begin{equation}\label{cond-mlwp}
-\frac16<s<0, \quad  -s\le r< \frac12 + 2s,
\end{equation}
Let $b=\frac12+ \varepsilon$ for sufficiently small $\varepsilon>0$
depending on $s,r$.
 Assume also that $A$ and $B$ in \eqref{u0v0p0} are such that
\begin{equation}\label{AB-cond}
C(B+A^2)(N^{-2\varepsilon}+ N^{-r+2\varepsilon})\le 1.
\end{equation}
Then there exists
\begin{equation}\label{DeltaT}
\Delta T\approx N^{(s-\varepsilon)/(r-2s-2\varepsilon)}
\end{equation}
such that \eqref{DKG2}, \eqref{data} has a unique solution
$$(u,v,\phi) \in X_+^{s, b}(S_{\Delta T})\times X_-^{s, b}(S_{\Delta
T}) \times \mathcal H^{r, b}(S_{\Delta T})$$ on the time interval
$0\le t\le \Delta T$. Moreover,
\begin{align}\label{control:Iu+Iv}
&\norm{Iu}_{X_+^{0, b}(S_{\Delta T})}+
\norm{Iv}_{X_-^{0,b}(S_{\Delta T})}\le CAN^{-s},\\
\label{control:Iphi}
 &\norm{I^2\phi}_{\mathcal H^{r-2s, b}(S_{\Delta
T})}\le C(B+A^2)N^{-2s},
\end{align}
where $C$ depends on $s$, $r$ and $\varepsilon$, but not $N$ or
$\Delta T$.
\end{theorem}

Combining \eqref{almostconser-IuIv}, \eqref{control:Iu+Iv} and
\eqref{control:Iphi} we conclude the following
 \emph{almost conservation
law}:
\begin{corollary}\label{Cor-almostconsv}
Let $s,r, \Delta T, \varepsilon, A, B, u$ and $v$ be as in Theorem
\ref{theorem-modifiedlwp}. Then
\begin{equation}\label{almostconser-IuIv:0N}
\norm{ Iu ( \Delta T) }_{L^2}^2  + \norm{Iv ( \Delta T) }_{L^2}^2
\le \norm{ Iu_0 }_{L^2}^2  + \norm{ Iv_0 }_{L^2}^2+
C(B+A^2)A^2N^{-r-2s+2\varepsilon}.
\end{equation}
\end{corollary}
As a consequence of this Corollary and \eqref{u0v0p0}, we obtain
\begin{equation}\label{almostconser-IuIv:N}
\norm{ Iu ( \Delta T) }_{L^2}^2  + \norm{Iv ( \Delta T) }_{L^2}^2
\le A^2N^{-2s}+ C(B+A^2)A^2N^{-r-2s+2\varepsilon}.
\end{equation}

We also need to control the growth of $I^2\phi$. To do so, we first
split $\phi$ into its homogeneous and inhomogeneous parts. Let
$\phi^{(0)}$ be solution of the homogenous Klein-Gordon Cauchy
problem
\begin{equation}\label{hom-eq-Iphi}
\left\{\begin{aligned}
&\left(\square -m^2\right)\phi^{(0)}=0\\
&\phi^{(0)}(0)=\phi_0, \quad
\partial_t\phi^{(0)}(0)=\phi_1.
\end{aligned}
\right.
\end{equation}
Then we write
$$
\phi= \phi^{(0)}+ \Phi
$$
where
\begin{equation}\label{Integ-Iphi:1}
\Phi= \left(\square-m^2\right)^{-1}(-2(\re(u\overline{v} ))).
\end{equation}
Here $\left(\square-m^2\right)^{-1}F$ denotes the solution of
$\left(\square-m^2\right) w = F$ with vanishing initial data.

The solution of the homogeneous Cauchy problem \eqref{hom-eq-Iphi}
in Fourier space is given by
\begin{equation}\label{soln-Iphi0}
\widehat{\phi^{(0)}(t)}(\xi)=\cos(t\angles{\xi}_m)\widehat{\phi_0}(\xi)+\frac{\sin(t\angles{\xi}_m)}{\angles{\xi}_m}\widehat{\phi_1}(\xi).
\end{equation}
Then by the energy estimate we have
\begin{equation}\label{energy-Iphi0}
\norm{ I^2\phi^{(0)}[ t] }_{H^{r-2s}}\le C(\norm{ I^2\phi_0
}_{H^{r-2s}}+ \norm{I^2\phi_1}_{H^{r-2s-1}}),
\end{equation}
for some $C>0$ and for all $t\ge 0$.

Now, consider the inhomogeneous part, \eqref{Integ-Iphi:1}. Since
the multiplier $q$ is assumed to be even, we obtain
$$I^2\re(u\overline v)=\re(I^2(u\overline v
))=\re(I(Iu\cdot I\overline v))+ \re(I Q_I( u, \overline v)). $$
Using this identity, we write
\begin{equation}\label{Inteq-Iphi:2}
I^2\Phi=\left (\square-m^2\right)^{-1}(-2\re(I(I u\cdot I \overline
v)))+(\square-m^2)^{-1}\left(-2\re( IQ_I( u, \overline v))\right).
\end{equation}
We then prove the following:
\begin{lemma}\label{lemma-almostconsve-phi}
 Suppose
\begin{equation}\label{cond-Iphi}
-1/4<s<0, \qquad 0<r<1/2+2s.
\end{equation}
Let $b=\frac12+ \varepsilon$ for sufficiently small $\varepsilon>0$
depending on $s,r$, and $\Delta T$ be as in Theorem
\ref{theorem-modifiedlwp}. Then
\begin{equation}\label{almostconser-Iphi}
\begin{split}
\norm{ I^2\Phi [ \Delta T] }_{H^{r-2s}}&\le C \Delta T(\norm{ Iu_0
}_{L^2}^2  + \norm{ Iv_0 }_{L^2}^2)\\
&\quad + C\Delta T N^{-r+2s+2\varepsilon}\norm{I^2\phi}_{H^{r-2s,
b}(S_{\Delta T})}\norm{Iu}_{X_+^{0, b}(S_{\Delta T})}
\norm{Iv}_{X_-^{0,b}(S_{\Delta T})} \\
&\quad + CN^{-1/2+2\varepsilon}\norm{Iu}_{X_+^{0, b}(S_{\Delta T})}
\norm{Iv}_{X_-^{0,b}(S_{\Delta T})} ,
\end{split}
\end{equation}
where $C$ depends on $s$, $r$, and $\varepsilon$, but not $N$ or
$\Delta T$.
\end{lemma}
 Then, by \eqref{u0v0p0}, \eqref{control:Iu+Iv}, \eqref{control:Iphi}
and \eqref{almostconser-Iphi} we conclude
\begin{corollary}\label{Cor-acl:Iphi-N} Let $A$,
$B$, $\Delta T$ be as in Theorem \ref{theorem-modifiedlwp} and $s,r,
\varepsilon$ be as in Lemma \ref{lemma-almostconsve-phi}. Then
\begin{equation}\label{almostconser-Iphi-N}
\norm{ I^2\Phi [ \Delta T] }_{H^{r-2s}}\le CA^2\left( \Delta
TN^{-2s} + (B+A^2)\Delta T N^{-r-2s+2\varepsilon}+
N^{-1/2-2s+2\varepsilon}\right).
\end{equation}
\end{corollary}
By \eqref{u0v0p0} and \eqref{energy-Iphi0}, we also have
\begin{equation}\label{energy-Iphi0-N}
\norm{ I^2\phi^{(0)}[ t] }_{H^{r-2s}}\le CBN^{-2s},
\end{equation}
for some $C>0$ and for all $t\ge 0$.

\section{Proof of Theorem \ref{mainthm} }
We first remark that by propagation of higher regularity (see
 Remark 1.4 in ~\cite{s2007} for the detail on this argument), it
suffices to prove Theorem \ref{mainthm} for $r < 1/2+2s$. We
therefore fix $s$ and $r$ satisfying
\begin{equation}\label{reducedcond-sr}
-\frac18<s<0, \quad s+\sqrt{s^2-s}< r< \frac12+2s.
\end{equation}
Observe that this region is contained in the intersection of the
regions in \eqref{cond-Qest}, \eqref{cond-mlwp} and
\eqref{cond-Iphi}, so the statements made in Theorem
\ref{theorem-modifiedlwp}, Lemmas \ref{lemma-errorestimate} and
 \ref{lemma-almostconsve-phi}, Corollaries \ref{Cor-almostconsv} and
\ref{Cor-acl:Iphi-N}, \eqref{almostconser-IuIv:N} and
\eqref{almostconser-Iphi-N} hold true for $s,r$ satisfying
\eqref{reducedcond-sr}.

Global well-posedness of \eqref{DKG2}, \eqref{data} will follow if
we show well-posedness on $[0,T]$ for arbitrary $0<T<\infty$. We
have already shown in Theorem \ref{theorem-modifiedlwp} that
\eqref{DKG2}, \eqref{data} is well-posed on $[0,\Delta T]$, where
the size of $\Delta T$ is given by \eqref{DeltaT}. Now, we divide
the interval $[0, T]$ into subintervals of length $\Delta T$. Let
$K$ be the number of subintervals, so
\begin{equation}\label{K}
K=\frac{T}{\Delta T}\approx
N^{(-s+\varepsilon)/(r-2s-2\varepsilon)}.
\end{equation}
To reach the given time $T$, we need to advance the solution from
$\Delta T$ to $2\Delta T$ etc. up to $K\Delta T$, successively.

We shall use induction argument to show well-posedness of
\eqref{DKG2}, \eqref{data} up to time $T$. We denote the solution of
\eqref{DKG2}, \eqref{data} on the $n$-th subinterval $[(n-1)\Delta
T, n\Delta T]$, where $1\le n\le K$, by $(u_n, v_n, \phi_n)$. Now,
consider the DKG system
\begin{equation}\label{DKG2:n}
\left\{
\begin{aligned}
  &iD_+u_n=Mu_n-\phi_n v_n,\\
  & iD_-v_n =Mv_n-\phi_n u_n,
   \\
  &\square\phi_n=m^2\phi_n -  2\re(u_n \overline v_n).
\end{aligned}
\right.
\end{equation}
The initial data for this system at  $t=(n-1)\Delta T$ is specified
by the induction scheme
\begin{equation}\label{data:n}
\left\{
\begin{aligned}
&u_n((n-1)\Delta T)= u_{n-1}((n-1)\Delta T)\in H^s, \\
& v_n((n-1)\Delta T)=v_{n-1}((n-1)\Delta T)\in H^s,\\
&\phi_n((n-1)\Delta T)= \phi_{n-1}((n-1)\Delta T)\in H^{r},\\
&\partial_t \phi_n((n-1)\Delta T)=\partial_t \phi_{n-1}((n-1)\Delta
T) \in H^{r-1}.
\end{aligned}
\right.
\end{equation}
The corresponding $I$-system will be
\begin{equation}\label{IDKG2:n}
\left\{
\begin{aligned}
  &iD_+(Iu_n)=MIu_n-I(\phi_n v_n),\\
  & iD_-(Iv_n) =MIv_n-I(\phi_n u_n),
   \\
  &\square(I^2\phi_n)=m^2I^2\phi_n -  2I^2(\re(u_n \overline v_n)),
\end{aligned}
\right.
\end{equation}
with the $I$-initial data at $t=(n-1)\Delta T$:
\begin{equation}\label{Idata:n}
\left\{
\begin{aligned}
&Iu_n((n-1)\Delta T)= Iu_{n-1}((n-1)\Delta T)\in L^2, \\
& Iv_n((n-1)\Delta T)=Iv_{n-1}((n-1)\Delta T)\in L^2,\\
&I^2\phi_n((n-1)\Delta T)= I^2\phi_{n-1}((n-1)\Delta T)\in H^{r-2s},\\
&\partial_t I^2\phi_n((n-1)\Delta T)=\partial_t
I^2\phi_{n-1}((n-1)\Delta T) \in H^{r-2s-1}.
\end{aligned}
\right.
\end{equation}
Note that for $n=1$, this $I$-initial value problem corresponds to
\eqref{IDKG2}, \eqref{Idata}.

In the following estimates and the rest of this section we shall use
the notation $$S_{n\Delta T}=[(n-1)\Delta T, n\Delta T]\times \R .$$
Recall that $(u_n, v_n, \phi_n)$ is a solution of DKG on the $n$-th
subinterval $[(n-1)\Delta T, n\Delta T]$ for given data at
$t=(n-1)\Delta T$. Then in view of \eqref{almostconser-IuIv} we have
\begin{equation}\label{almostconser-IuIv-n}
\begin{split}
&\norm{ Iu_n ( n\Delta T) }_{L^2}^2  + \norm{Iv_n ( n\Delta T) }_{L^2}^2\\
&\le \norm{ Iu_n( (n-1)\Delta T)}_{L^2}^2  + \norm{ Iv_n( (n-1)\Delta T) }_{L^2}^2\\
&\quad + CN^{-r+2s+2\varepsilon}\norm{I^2\phi_n}_{H^{r-2s,
b}(S_{n\Delta T})}\norm{Iu_n}_{X_+^{0, b}(S_{n\Delta T})}
\norm{Iv_n}_{X_-^{0,b}(S_{n\Delta T})}.
\end{split}
\end{equation}
On the other hand, splitting $\phi_n$ into its homogeneous and
inhomogeneous parts, $\phi_n= \phi_n^{(0)}+\Phi_n$, we have in view
of \eqref{energy-Iphi0} and \eqref{almostconser-Iphi}
\begin{equation}\label{almostconser-Iphi-n}
\begin{split}
&\norm{ I^2\Phi_n [ n\Delta T] }_{H^{r-2s}}\\
&\le C \Delta T(\norm{ Iu_n( (n-1)\Delta T)}_{L^2}^2  + \norm{ Iv_n( (n-1)\Delta T) }_{L^2}^2)\\
&\quad + C\Delta T N^{-r+2s+2\varepsilon}\norm{I^2\phi_n}_{H^{r-2s,
b}(S_{n\Delta T})}\norm{Iu_n}_{X_+^{0, b}(S_{n\Delta T})}
\norm{Iv_n}_{X_-^{0,b}(S_{n\Delta T})} \\
&\quad + CN^{-1/2+2\varepsilon}\norm{Iu_n}_{X_+^{0, b}(S_{n\Delta
T})} \norm{Iv_n}_{X_-^{0,b}(S_{n\Delta T})},
\end{split}
\end{equation}
and
\begin{equation}\label{energy-Iphi0-n}
\sup_{0\le t\le T}\norm{ I^2\phi_n^{(0)}[ t] }_{H^{r-2s}}\le C\norm{
I^2\phi_n[(n-1)\Delta T] }_{H^{r-2s}}.
\end{equation}

Our induction hypotheses will be
\begin{align}\label{Indhyp:uv}
&\norm{ Iu_n ( (n-1)\Delta T) }_{L^2}  + \norm{Iv_n ( (n-1)\Delta
T) }_{L^2} \le A_nN^{-s},\\
\label{Indhyp:phi} &\norm{ I^2\phi_n [(n-1)\Delta T]}_{H^{r-2s}}\le
B_nN^{-2s},\end{align} for some $1\le n<K$, where $A_n$ and $B_n$
are independent of $N$. Again, at the first induction step, $n=1$,
\eqref{Indhyp:uv} and \eqref{Indhyp:phi} hold by \eqref{u0v0p0}.
Now, by Theorem \ref{theorem-modifiedlwp} we know that $(u_n, v_n,
\phi_n)$ solves \eqref{DKG2:n}, \eqref{data:n} on the $n$-th
subinterval $[(n-1)\Delta T, n\Delta T]$, where the size of $\Delta
T$ is given by \eqref{DeltaT}, provided that the boot-strap
condition
\begin{equation}\label{AB-cond:n}
C(B_n+A_n^2)(N^{-2\varepsilon}+ N^{-r+2\varepsilon})\le 1
\end{equation}
is satisfied. Moreover, these solutions satisfy the bound
\begin{align}\label{control:IuIv-n}
&\norm{Iu_n}_{X_+^{0, b}(S_{n\Delta T})}+
\norm{Iv_n}_{X_-^{0,b}(S_{n\Delta T})}\le CA_nN^{-s},\\
\label{control:Iphi-n}
 &\norm{I^2\phi_n}_{\mathcal H^{r-2s, b}(S_{n\Delta
T})}\le C(B_n+A_n^2)N^{-2s}.
\end{align}
So, if we can prove that $A_n$ and $B_n$ stay bounded for all $1\le
n\le K$, then \eqref{AB-cond:n} will be satisfied for all $1\le n\le
K$, choosing $\varepsilon$ small enough and $N$ large enough (recall
$r>0$). We can therefore apply Theorem \ref{theorem-modifiedlwp} $K$
times, and hence prove well-posedness on $[0, T]$.

 By \eqref{control:IuIv-n},
\eqref{control:Iphi-n} and the induction hypotheses
\eqref{Indhyp:uv} and \eqref{Indhyp:phi}, the estimates
\eqref{almostconser-IuIv-n} and \eqref{almostconser-Iphi-n} imply
\begin{align}\label{almostconser-IuIv:Nn}
&\norm{ Iu_n ( n\Delta T) }_{L^2}^2  + \norm{Iv_n ( n\Delta T)
}_{L^2}^2 \le A_n^2N^{-2s}+ C(B_n+A_n^2)A_n^2N^{-r-2s+2\varepsilon},
\\
\label{almostconser-Iphi:Nn} &\norm{ I^2\Phi_n [ n\Delta T]
}_{H^{r-2s}}\le CA_n^2\left( \Delta TN^{-2s} + (B_n+A_n^2)\Delta T
N^{-r-2s+2\varepsilon}+ N^{-1/2-2s+2\varepsilon}\right),
\end{align}
whereas \eqref{energy-Iphi0-n} and \eqref{Indhyp:phi} imply
\begin{equation}\label{energy-Iphi0-Nn}
\sup_{0\le t\le T}\norm{ I^2\phi_n^{(0)}[ t] }_{H^{r-2s}}\le
CB_nN^{-2s}.
\end{equation}

By \eqref{Idata:n} and \eqref{almostconser-IuIv:Nn} we obtain
\begin{align*}
\norm{ Iu_{n+1} ( n\Delta T) }_{L^2}^2  + \norm{Iv_{n+1} ( n\Delta
T) }_{L^2}^2&=\norm{ Iu_n ( n\Delta T) }_{L^2}^2  + \norm{Iv_n (
n\Delta T) }_{L^2}^2 \\
&\le A_n^2N^{-2s}+ C(B_n+A_n^2)A_n^2N^{-r-2s+2\varepsilon}.
\end{align*}
We therefore have
\begin{equation}\label{An+1}
 A_{n+1}^2\le A_n^2+ C(B_n+A_n^2)A_n^2N^{-r+2\varepsilon}.
\end{equation}

On the other hand, by \eqref{Idata:n}, \eqref{almostconser-Iphi:Nn}
and \eqref{energy-Iphi0-Nn} we get
\begin{align*}
\norm{ I^2\phi_{n+1} [n\Delta T] }_{H^{r-2s}}&=\norm{ I^2\phi_{n}
[n\Delta T] }_{H^{r-2s}}\\
&\le \norm{ I^2\phi_n^{(0)}[ n\Delta T] }_{H^{r-2s}}+\norm{
I^2\Phi_n [
n\Delta T] }_{H^{r-2s}}\\
&\le CB_nN^{-2s} + CA_n^2\left( \Delta TN^{-2s} + (B_n+A_n^2)\Delta
T N^{-r-2s+2\varepsilon}+ N^{-1/2-2s+2\varepsilon}\right)
\end{align*}
Therefore,
\begin{equation}\label{Bn+1-b}
 B_{n+1}\le  CB_n + CA_n^2 \Delta T + C(B_n+A_n^2)A_n^2\Delta T
N^{-r+2\varepsilon}+CA_n^2 N^{-1/2+2\varepsilon}.
\end{equation}
However, the presence of a constant $C$ in front of $B_n$ in the
first term of the r.h.s. of this inequality is bad, since then $B_n$
will grow exponentially in $n$; after $n$ induction steps, $
B_n\approx C^n$. To fix this problem, we follow ~\cite{s2007} to
write $\phi_n^{(0)}$ as a cascade of free waves:
$$ \phi^{(0)}_{n+1}=\phi^{(0)}_1 + \tilde{\phi}^{(0)}_2 + \cdot \cdot \cdot +\tilde{\phi}^{(0)}_n +  \tilde{\phi}^{(0)}_{n+1},$$
for $n\ge 1$, where
\begin{equation}\label{Iter-Iphi} \left\{\begin{aligned}
&\left(\square -m^2\right)\tilde{\phi}^{(0)}_{n+1}=0\\
&\tilde{\phi}^{(0)}_{n+1}( n\Delta T)=\Phi_{n}(n\Delta T), \\
&\partial_t\tilde{\phi}^{(0)}_{n+1}(n\Delta T)=\partial_t\Phi_n(
n\Delta T).
\end{aligned}
\right.
\end{equation}
Now, by energy inequality and \eqref{almostconser-Iphi:Nn} we have
\begin{equation}\label{almostconser-Itiphi:Nn}
\norm{ I^2\tilde{\phi}^{(0)}_{n+1} [ t] }_{H^{r-2s}}\le CA_n^2\left(
\Delta TN^{-2s} + (B_n+A_n^2)\Delta T N^{-r-2s+2\varepsilon}+
N^{-1/2-2s+2\varepsilon}\right),
\end{equation}
in the entire time interval $0\le t\le T$.

We now replace the induction hypothesis \eqref{Indhyp:phi} by the
stronger condition \begin{equation}\label{mod-indhyp:phi} \sup_{0\le
t\le T}\norm{ I^2\phi_n^{(0)}[t]}_{H^{r-2s}}\le B_nN^{-2s}.
\end{equation}
Since $\phi_{n+1}^{(0)}=\phi_{n}^{(0)}+ \tilde{\phi}_{n+1}^{(0)}$,
we have
$$
\norm{ I^2\phi_{n+1}^{(0)} [ t] }_{H^{r-2s}}\le \norm{
I^2\phi_n^{(0)}[t] }_{H^{r-2s}}+ \norm{I^2\tilde{\phi}^{(0)}_{n+1} [
t] }_{H^{r-2s}},
$$
 for all $0\le t\le T$. Then using \eqref{mod-indhyp:phi} and
 \eqref{almostconser-Itiphi:Nn}, we conclude
\begin{equation}\label{Bn+1-mod}
 B_{n+1}\le  B_n + CA_n^2 \Delta T + C(B_n+A_n^2)A_n^2\Delta T
N^{-r+2\varepsilon}+CA_n^2 N^{-1/2+2\varepsilon}.
\end{equation}
This estimate will be a replacement for the ``bad'' estimate
\eqref{Bn+1-b}.

 Now, we claim that if $\varepsilon > 0$ is chosen small enough,
and then $N$ large enough, depending on $\varepsilon$, then for $1
\le n \le K$,
\begin{equation}\label{claim-AB}
A_n\le \rho \equiv 2A_1, \qquad B_n\le  \sigma \equiv 2B_1+
4CTA_1^2.
\end{equation}
We proceed by induction. Assume that \eqref{claim-AB} holds for
$1\le n<k$, for some $k\le K.$ Then \eqref{AB-cond:n} reduces to
\begin{equation}\label{AB-condredu:n}
C(\sigma+\rho^2)(N^{-2\varepsilon}+ N^{-r+2\varepsilon})\le 1,
\end{equation}
for $n<k$. Since $r>0$, we can choose $\varepsilon$ very small and
$N$ very large to ensure that \eqref{AB-condredu:n} is satisfied. So
by \eqref{An+1} and \eqref{Bn+1-mod}, and the assumption that
\eqref{claim-AB} holds for $n < k$, we get (for $n < k$)
\begin{align*}
 &A_{n+1}^2\le A_1^2+ nC\sigma\rho^2N^{-r+2\varepsilon}\\
 &B_{n+1}\le  B_1 + n\left[C\rho^2 \Delta T + C\sigma\rho^2\Delta T
N^{-r+2\varepsilon}+  C\rho^2 N^{-1/2+2\varepsilon}\right].
\end{align*}
Furthermore, \eqref{claim-AB} will be satisfied for $A_k$ and $B_k$
provided that
\begin{align*}
 &(k-1)C\sigma\rho^2N^{-r+2\varepsilon}\le 3A_1^2\\
 &(k-1)\left(C\rho^2 \Delta T + C\sigma\rho^2\Delta T
N^{-r+2\varepsilon}+ C\rho^2 N^{-1/2+2\varepsilon}\right)\le
B_1+4CTA_1^2.
\end{align*}
Now, since $k\le K= T/(\Delta T)\le
CN^{(-s+\varepsilon)/(r-2s-2\varepsilon)} $, by \eqref{DeltaT}, it
suffices to have
\begin{align}
\label{cond:sr-1}
 &C\sigma\rho^2N^{(-s+\varepsilon)/(r-2s-2\varepsilon)-r+2\varepsilon}\le 3A_1^2,\\
 \label{cond:sr-2}
&CT\sigma\rho^2 N^{-r+2\varepsilon}\le B_1/2,
 \\
 \label{cond:sr-3}
 &C\rho^2 N^{(-s+\varepsilon)/(r-2s-2\varepsilon)-1/2+2\varepsilon}\le
 B_1/2,\\
 \label{cond:sr-4}
 &CT\rho^2\le 4CTA_1^2.
\end{align}
Here, to get the l.h.s. of \eqref{cond:sr-4} we used the fact that
$(k-1)\Delta T\le K\Delta T=T$; In fact, \eqref{cond:sr-4} holds
with equality, since $\rho=2A$. Since $r>0$, \eqref{cond:sr-2} will
be satisfied by choosing first $\varepsilon$ small enough and then
$N$ sufficiently large. To satisfy \eqref{cond:sr-1} and
\eqref{cond:sr-3}, it suffices to have
\begin{equation}\label{cond:sr-final}
\frac{-s+\varepsilon}{r-2s-2\varepsilon}-r+2\varepsilon<0, \qquad
\frac{-s+\varepsilon}{r-2s-2\varepsilon}-1/2+2\varepsilon<0.
\end{equation}
The first condition is equivalent to
$r^2-2sr+s>\varepsilon(4(r-s)+1-4\varepsilon).$ Choosing
$\varepsilon>0$ very small, this reduces to $r^2-2sr+s>0,$ i.e.,
$r>s+\sqrt{s^2-s}$, which holds by assumption
\eqref{reducedcond-sr}. The second condition in
\eqref{cond:sr-final} is weaker than the first condition since by
assumption \eqref{reducedcond-sr}, $r<1/2+2s$ and $s<0$.

Thus, \eqref{claim-AB} holds for $n=1, \cdot \cdot \cdot, K$, and
hence the proof is complete.

\section{Proof of Lemma \ref{lemma-errorestimate}}\label{sec-error-estimate}

Taking the Fourier transform in space, we get
\begin{equation}\label{FT:QI}
[Q_I(f,g)]\FT(\xi)=\int [q(\xi)-q(\eta)q(\xi-\eta)]\hat f(\eta)\hat
g(\xi-\eta)d\eta.
\end{equation}
Recall that the symbol $q(\zeta)=1$ for $\abs{\zeta}<N$.

We now write $u =  u_l + u_h $, $v=  v_l + v_h $, $\phi =  \phi_l +
\phi_h $ with ${\widehat{u_l}}, \ {\widehat{v_l}}, \
{\widehat{\phi_l}}$ supported on $\{ \xi : |\xi | \ll N \}$ and
${\widehat{u_h}}, \ {\widehat{v_h}}, \ {\widehat{\phi_h}}$ supported
on $\{ \xi : |\xi | \gtrsim N \}$. Since we are considering
(weighted) $L^2$ norms, we can replace ${\widehat{u}}$,
${\widehat{v}}$ and ${\widehat{\phi}}$ by $|{\widehat{u}}|$,
$|{\widehat{v}}|$ and $|{\widehat{\phi}}|$. Assume therefore that
${\widehat{u}} , {\widehat{v}}, {\widehat{\phi}} \geq 0$.

We only prove \eqref{Q:phi-u} since the proof for \eqref{Q:phi-v} is
quite similar. The only difference is that to prove \eqref{Q:phi-u},
we use the product estimate \eqref{prodembed:-+}, but to prove
\eqref{Q:phi-v}, we use \eqref{prodembed:+-}. We prove
\eqref{Q:phi-u} for all possible interactions. As a matter of
convenience we skip the time restriction in this section.

\subsection{Low/low interaction}
Recalling \eqref{FT:QI}, we have
$$[Q_I(\phi_{l}, u_l)]\FT(\xi)=\int
[q(\xi)-q(\eta)q(\xi-\eta)]\hat \phi_{l}(\eta)\hat
u_{l}(\xi-\eta)d\eta.$$ But since $\abs{\eta},\abs{\xi-\eta}\ll N$,
which in turn implies $\abs{\xi}< N$, the expression inside the
square bracket in the above integral vanishes.

\subsection{Low/high interaction} Then
$$[Q_I(\phi_{l}, u_h)]\FT(\xi)=\int
[q(\xi)-q(\xi-\eta)]\hat \phi_{l}(\eta)\hat u_{h}(\xi-\eta)d\eta,$$
because $q(\eta)=1$ on the support of $\hat \phi_{l}$.  By the mean
value theorem,
$$
\abs{q(\xi)-q(\xi-\eta)}\le \abs{q'(\zeta)}\abs{\eta},$$ where
$\zeta$ lies between $\xi$ and $\xi-\eta$.

 Now, assume
$\abs{\xi-\eta}\gg N$. Then $\abs{\eta}\ll \abs{\xi-\eta}$, and this
implies
$$\abs{\xi} \approx \abs{\xi-\eta}\approx \abs{\zeta}.$$ Hence
$$\abs{q'(\zeta)}=N^{-s}\abs{s\abs{\zeta}^{s-1}}\approx
N^{-s}\abs{s\abs{\xi-\eta}^{s-1}} $$ Next, assume
$\abs{\xi-\eta}\approx N$. If $\abs{\zeta}< N$, then $q'(\zeta)=0$.
If $\abs{\zeta}>2 N$, then
$$\abs{q'(\zeta)}=N^{-s}\abs{s\abs{\zeta}^{s-1}}\lesssim
N^{-s}\abs{\xi-\eta}^{s-1}.$$ Finally, assume $N\le \abs{\zeta}\le
2N$. In this case, we define $q(\xi)=\chi(\xi/N)$ where $\chi$ is a
smooth, even and monotone function defined by
\begin{equation*}
 \chi( \sigma )=
\begin{cases}
1 & \text{if $0\le \sigma<1$}, \\
\sigma^s & \text{if $ \sigma>2$}.
\end{cases}
\end{equation*}
Then
$$\abs{q'(\zeta)}\lesssim N^{-1}\lesssim
N^{s}\abs{\xi-\eta}^{s-1}.$$ We therefore conclude
$$
\abs{q(\xi)-q(\xi-\eta)}\lesssim
N^{-s}\abs{\xi-\eta}^{s-1}\abs{\eta}.
$$
Interpolating this with the  trivial estimate
$$
\abs{q(\xi)-q(\xi-\eta)}\lesssim N^{-s}\abs{\xi-\eta}^{s}
$$
we get
$$
\abs{q(\xi)-q(\xi-\eta)}\lesssim
N^{-s}\abs{\xi-\eta}^{s}\abs{\xi-\eta}^{-\theta}\abs{\eta}^{\theta},
$$
for $ \theta \in [0,1]$.

Then
\begin{equation}\label{mvt-app}
\begin{split}
\abs{[Q_I(\phi_{l}, u_h)]\FT(\xi)}&\lesssim \int
\abs{\eta}^{\theta}\hat \phi_{l}(\eta) \abs{\xi-\eta}^{-\theta}
N^{-s}\abs{\xi-\eta}^{s}\hat u_{h}(\xi-\eta)d\eta\\
&\lesssim[D^\theta\phi_{l}\cdot D^{-\theta}Iu_h]\FT(\xi) .
\end{split}
\end{equation}

Now, choosing $\theta= r-2s$ and applying the product estimate
\eqref{prodembed:-+}, we get
\begin{align*}
\norm{Q_I(\phi_{l},
u_h)}_{X_-^{0,-b}}&\lesssim\norm{D^{r-2s}\phi_{l}\cdot
D^{-r+2s}Iu_h}_{X_-^{0,-b}}\\&\lesssim\norm{D^{r-2s}\phi_{l}}_{
H^{0,
  b}}
\norm{D^{-r+2s}Iu_h}_{ X_+^{2\varepsilon, b}
}\\
  &\lesssim N^{-r+2s+2\varepsilon}\norm{\phi_{l}}_{
H^{r-2s,
  b}}\norm{Iu_h}_{ X_+^{0,
b} }.
\end{align*}
\subsection{High/low interaction}
A calculation similar to the preceding low/high interaction estimate
gives
\begin{align*}
\abs{[Q_I(\phi_h, u_{l})]\FT(\xi)}&\lesssim[D^{-\theta}I\phi_h\cdot
D^{\theta}u_{l}]\FT(\xi) .
\end{align*}
Take $\theta=0$. Applying the product estimate \eqref{prodembed:-+}
and \eqref{compare-Iz:z}, we get
\begin{align*}
\norm{Q_I(\phi_h, u_{l})}_{X_-^{0,-b}}&\lesssim\norm{I\phi_h\cdot
u_{l}}_{X_-^{0,-b}}\\&\lesssim\norm{I\phi_{h}}_{ H^{2\varepsilon,
  b}} \norm{u_{l}}_{ X_+^{0,
b}
}\\
&\lesssim N^{-r+s+2\varepsilon}\norm{I\phi_{h}}_{ H^{r-s,
  b}} \norm{u_{l}}_{ X_+^{0,
b} },\\
  &\lesssim N^{-r+2s+2\varepsilon}\norm{I^2\phi_{h}}_{
H^{r-2s,
  b}} \norm{u_{l}}_{ X_+^{0,
b} }.
\end{align*}
\subsection{High/high interaction}
Here, we do not take advantage of any cancellation. We instead use
the triangle inequality to get
\begin{align*}
\norm{Q_I(\phi_h, u_{h})}_{X_-^{0,-b}}\le \norm{I(\phi_h
u_{h})}_{X_-^{0,-b}}+ \norm{ I\phi_h\cdot Iu_{h}}_{ X_-^{0, -b} }.
\end{align*}
By \eqref{Ibdd-1}, the product estimate \eqref{prodembed:-+}, and
\eqref{compare-Iz:z}, we get
\begin{align*}
\norm{I(\phi_h u_h)}_{X_-^{0, -b}}&\lesssim \norm{\phi_h
u_h}_{X_-^{0, -b}}\\
&\lesssim \norm{\phi_h}_{H^{-s+2\varepsilon, b}}\norm{u_h}_{X_+^{s,
b}}\\
&= \norm{\phi_h}_{H^{r-r-s+2\varepsilon, b}} N^{s}
N^{-s}\norm{u_h}_{X_+^{s,
b}}\\
&\lesssim N^{-r-s+2\varepsilon}\norm{\phi_h}_{H^{r-s+s, b}} N^{s}
\norm{Iu_h}_{X_+^{0, b}}\\
&= N^{-r+s+2\varepsilon} \norm{I\phi_h}_{H^{r-s,
b}}\norm{Iu_h}_{X_+^{0, b}},\\
&\lesssim N^{-r+2s+2\varepsilon} \norm{I^2\phi_h}_{H^{r-2s,
b}}\norm{Iu_h}_{X_+^{0, b}},
\end{align*}
and
\begin{align*}
\norm{I\phi_h\cdot Iu_{h}}_{X_-^{0,-b}}&\lesssim\norm{I\phi_{h}}_{
H^{2\varepsilon,
 b}} \norm{
Iu_{h}}_{ X_+^{0,b}
}\\
  &\lesssim N^{-r+s+2\varepsilon}\norm{I\phi_{h}}_{
H^{r-s,
  b}} \norm{Iu_{h}}_{ X_+^{0,
b} }\\
 &\lesssim N^{-r+2s+2\varepsilon}\norm{I^2\phi_{h}}_{
H^{r-2s,
  b}} \norm{Iu_{h}}_{ X_+^{0,
b} } .
\end{align*}

\section{Proof of Lemma \ref{lemma-almostconsve-phi} }
First, we estimate the first term in the right hand side of
\eqref{Inteq-Iphi:2}. By energy inequality, \eqref{Ibdd-1}, Lemma
\ref{Sob-emb-thm} and \eqref{almostconser-IuIv}  we get (recall that
$r<1/2+2s$)
\begin{equation}\label{energy-Iphi1}
\begin{split}
&\norm{ \left(\square-m^2\right)^{-1}2\re(I(I u \cdot I\overline v))
[
\Delta T] }_{H^{r-2s}}\\
&\le C\int_0^{\Delta T}\norm{\re(I(Iu(t)\cdot I\overline
v(t)))}_{H^{r-2s-1}}dt\\
&\le C\int_0^{\Delta T}\norm{Iu(t)\cdot I\overline
v(t)}_{H^{r-2s-1}}dt\\
&\le
C\int_0^{\Delta T}\norm{Iu(t)}_{L^2}\norm{Iv(t)}_{L^2}dt\\
&\le
C\int_0^{\Delta T}\norm{Iu(t)}_{L^2}^2+\norm{Iv(t)}_{L^2}^2dt\\
&\le C\Delta T
\left(\norm{Iu_0}_{L^2}^2+\norm{Iv_0}_{L^2}^2\right)\\
&\quad + C\Delta TN^{-r+2s+2\varepsilon}\norm{I^2\phi}_{H^{r-2s,
b}(S_{\Delta T})}\norm{Iu}_{X_+^{0, b}(S_{\Delta T})}
\norm{Iv}_{X_-^{0,b}(S_{\Delta T})}.
\end{split}
\end{equation}

Now, we estimate the second term in the right hand side of
\eqref{Inteq-Iphi:2}. We claim that
\begin{equation}
\label{energy-Q:u-v}
\begin{split}
&\norm{\left (\square-m^2\right)^{-1}2\re (IQ_I( u,\overline v))
[ \Delta T] }_{H^{r-2s}}\\
&\le CN^{-1/2+2\varepsilon}\norm{Iu}_{X_+^{0, b}(S_{\Delta T})}
  \norm{ Iv}_{X_-^{0, b}(S_{\Delta T})}.
  \end{split}
\end{equation}
Assume for the moment that this claim is true. Then a combination of
the estimates \eqref{Inteq-Iphi:2}, \eqref{energy-Iphi1} and
\eqref{energy-Q:u-v} proves the Lemma.

It remains to prove the claim, \eqref{energy-Q:u-v}. By
\eqref{embxhc}, Lemma \ref{KGEElemma} and \eqref{Ibdd-1}
\begin{align*}
  &\norm{ \left(\square-m^2\right)^{-1}\re (IQ_I( u,\overline v)) [ \Delta T]
}_{H^{r-2s}}\\
&\le C  \norm{\left( \square-m^2\right)^{-1}\re (IQ_I( u, \overline
v))
}_{ H^{r-2s, b} (S_{\Delta T})}\\
&\le C \norm{IQ_I( u,\overline v)}_{H^{r-2s-1, b-1} (S_{\Delta
T})}\\
&\le C \norm{Q_I( u,\overline v)}_{H^{r-2s-1, b-1} (S_{\Delta T})}.
\end{align*}
Then to estimate $ \norm{Q_I( u, \overline v)}_{H^{r-2s-1,
b-1}(S_{\Delta T})}$ we follow a similar argument as in the
preceding subsection. As a matter of convenience we skip the time
restriction in the rest of the section. The contribution from the
low/low frequency interaction, $Q_I(u_l,v_l)$, vanishes by the same
argument as in the low/low frequency case in the preceding section.
For the low/high frequency case we use \eqref{mvt-app} with
$\theta=0$ (the high/low frequency case is similar) to get
$$
\abs{[Q_I( u_{l}, \overline {v_h})]\FT(\xi)}\lesssim \abs{[
u_{l}\cdot I\overline {v_h}]\FT(\xi)} .
$$
Then by \eqref{reduc-embedding:1},
\begin{align*}
\norm{ u_l\cdot I\overline{v_h} }_{H^{r-2s-1,b-1}} &\lesssim\norm{
u_l }_{X_+^{0,b}}
\norm{Iv_h}_{X_-^{-1/2+2\varepsilon,b}}\\
  &\lesssim N^{-1/2+2\varepsilon}\norm{u_{l}}_{
X_+^{0,
  b}}\norm{Iv_h}_{ X_-^{0,
b} }.
\end{align*}
To estimate the contribution from high/high interaction, we first
use the triangle inequality to get
$$
  \norm{Q_I( u_h,\overline{ v_h})}_{H^{r-2s-1,
b-1}}\le \norm{I( u_h \overline{ v_h})}_{H^{r-2s-1, b-1}}+ \norm{I
u_h\cdot I\overline v_h}_{H^{r-2s-1, b-1}}.
$$
Then applying \eqref{reduc-embedding:1}, we obtain
\begin{align*}
\norm{ I( u_h \overline{ v_h}) }_{H^{r-2s-1,b-1}} &\lesssim \norm{
u_h \overline{v_h} }_{H^{r-2s-1,b-1}}
\\
&\lesssim \norm{ u_h }_{X_+^{-1/4,b}}
\norm{v_h}_{X_-^{-1/4+2\varepsilon,b}}\\
  &\lesssim N^{-1/4-s}\norm{u_{h}}_{ X_+^{s,
  b}}N^{-1/4-s+2\varepsilon}\norm{v_h}_{
X_-^{s, b} }
\\
&\lesssim N^{-1/2+2\varepsilon}\norm{Iu_{h}}_{ X_+^{0,
  b}}\norm{Iv_h}_{ X_-^{0,
b} },
\end{align*}
and
\begin{align*}
\norm{ Iu_h\cdot I\overline{ v_h} }_{H^{r-2s-1,b-1}} &\lesssim
\norm{ Iu_h }_{X_+^{-1/4,b}}
\norm{Iv_h}_{X_-^{-1/4+2\varepsilon,b}}\\
&\lesssim N^{-1/2+2\varepsilon}\norm{Iu_{h}}_{ X_+^{0,
  b}}\norm{Iv_h}_{ X_-^{0,
b} }.
\end{align*}

\section{proof of Theorem \ref{theorem-modifiedlwp}}
Assume $0<\Delta T < 1$.
 Define
\begin{align*}
&\norm{Iw}_{X^{0,b}(S_{\Delta T})}= \norm{Iu}_{X_+^{0, b}(S_{\Delta
T})}+ \norm{Iv}_{X_-^{0,b}(S_{\Delta
T})},\\
&\norm{Iw_0}_{L^2}= \norm{Iu_0}_{L^2}+ \norm{Iv_0}_{L^2}.
\end{align*}
Applying Lemma \ref{DEElemma} to the first two equations and Lemma
\ref{KGEElemma} to the third equation of the $I$-system
\eqref{IDKG2}, we get
\begin{align*}
&\norm{Iu}_{X_+^{0, b}(S_{\Delta T})}\le C\left\{ \norm{Iu_0}_{L^2}+
\norm{I(\phi v)}_{X_+^{0, b-1}(S_{\Delta
T})}\right\},\\
&\norm{Iv}_{X_-^{0,b}(S_{\Delta T})}\le C\left\{ \norm{Iv_0}_{L^2}+
\norm{I(\phi u)}_{X_-^{0, b-1}(S_{\Delta
T})}\right\},\\
&\norm{I^2\phi}_{\mathcal H^{r-2s, b}(S_{\Delta T})}\le C\left\{
\norm{I^2\phi[0]}_{H^{r-2s}}+\norm{I^2(u \overline v)}_{H^{r-2s-1,
b-1 }(S_{\Delta T})}\right\}.
\end{align*}
   Now, we claim the following:
\begin{align}
\label{Ibilinear:phi-u}
 \norm{I(\phi u)}_{X_-^{0,
b-1}(S_{\Delta T})}&\le C\Gamma_1 \norm{I^2\phi}_{H^{r-2s,
b}(S_{\Delta T})}\norm{Iu}_{X_+^{0, b}(S_{\Delta T})},
\\
 \label{Ibilinear:phi-v}
\norm{I(\phi v)}_{X_+^{0, b-1}(S_{\Delta T})}&\le
C\Gamma_1\norm{I^2\phi}_{H^{r-2s, b}(S_{\Delta
T})}\norm{Iv}_{X_-^{0, b}(S_{\Delta T})} ,
\\
 \label{Ibilinear:u-v}
\norm{I^2( u\overline v)}_{H^{r-2s-1, b-1}(S_{\Delta T})}&\le
C\Gamma_2\norm{Iu}_{X_+^{0, b}(S_{\Delta T})}\norm{Iv}_{X_-^{0,
b}(S_{\Delta T})} ,
\end{align}
where
\begin{align*}
\Gamma_1=\Gamma_1(N, \Delta T):&=(\Delta T)^{2r-4s-4\varepsilon}+
N^{-r+2s+2\varepsilon},\\
\Gamma_2=\Gamma_2(N, \Delta T):&=(\Delta T)^{1-4\varepsilon}+
N^{-1/2+2\varepsilon}.
\end{align*}
 Assume for the moment that the claim is true. Then
\begin{align}\label{est:Iw}
\norm{Iw}_{X^{0, b}(S_{\Delta T})}&\le C\norm{Iw_0}_{L^2}+
C\Gamma_1\norm{I^2\phi}_{H^{r-2s, b}(S_{\Delta T})}\norm{Iw}_{X^{0,
b}(S_{\Delta T})},
\\
\label{est:Iphi} \norm{I^2\phi}_{\mathcal H^{r-2s, b}(S_{\Delta
T})}&\le C \norm{I^2\phi[0]}_{H^{r-2s}}+ C\Gamma_2\norm{Iw}^2_{X^{0,
b}(S_{\Delta T})}.
\end{align}
Using \eqref{est:Iphi}, the estimate \eqref{est:Iw} reduces to
\begin{equation}\label{est:Iw-reduc}
\begin{split}
\norm{Iw}_{X^{0, b}(S_{\Delta T})}&\le C\norm{Iw_0}_{L^2}+
C\Gamma_1\norm{I^2\phi[0]}_{H^{r-2s}}\norm{Iw}_{X^{0, b}(S_{\Delta
T})}+C\Gamma_1\Gamma_2\norm{Iw}^3_{X^{0, b}(S_{\Delta T})}\\
&\le CAN^{-s} + CBN^{-2s}\Gamma_1\norm{Iw}_{X^{0, b}(S_{\Delta T})}
+ C\Gamma_1\Gamma_2\norm{Iw}^3_{X^{0, b}(S_{\Delta T})}.
\end{split}
\end{equation}
So if
\begin{equation}\label{T-1stcond}
CBN^{-2s}\Gamma_1 (2CAN^{-s})+ C\Gamma_1\Gamma_2(2CAN^{-s})^3\le
CAN^{-s},
\end{equation}
then it follows by a boot-strap argument (see the
 Remark below for the detail on this argument) that
\begin{equation}\label{Iw-bdd}
\norm{Iw}_{X^{0, b}(S_{\Delta T})}\le 2CAN^{-s}.
\end{equation}
Now, if we choose
\begin{equation}\label{T-2ndcond}
\Delta T\approx  N^{(s-\varepsilon)/(r-2s-2\varepsilon)},
\end{equation}
 the boot-strap condition \eqref{T-1stcond} reduces to
(modifying $C$)
\begin{equation}\label{T-reducedcond}
C(B+A^2)\left(N^{-2\varepsilon}+ N^{-r+2\varepsilon}\right) \le 1.
\end{equation}
On the other hand, by \eqref{est:Iphi} we get (modifying $C$)
$$
\norm{I^2\phi}_{\mathcal H^{r-2s, b}(S_{\Delta T})}\le CBN^{-2s} +
4CA^2N^{-2s} \left(
N^{(s-\varepsilon)(1-4\varepsilon)/(r-2s-2\varepsilon)} +
N^{-1/2+2\varepsilon}\right).
$$
The second term in the r.h.s. of this inequality can be bounded by
$C(B+A^2)N^{-2s}$ since the quantity in the bracket is very small.
So, we obtain
\begin{equation}\label{est:Iphi-redu}
\norm{I^2\phi}_{\mathcal H^{r-2s, b}(S_{\Delta T})}\le
2C(B+A^2)N^{-2s}.
\end{equation}

\begin{remark}\label{bootstrap}
The above estimates imply LWP of \eqref{DKG2}, \eqref{data} with
time of existence up to $\Delta T>0$ given by \eqref{T-2ndcond}
provided that the condition \eqref{T-reducedcond} is satisfied. The
boot-strap argument mentioned above can be shown using the standard
iteration argument: Set $u^{(-1)}=v^{(-1)}=0$, and define for $n\ge
-1$ inductively
\begin{equation}\label{IDKG3}
\left\{
\begin{aligned}
  &iD_+(Iu^{(n+1)})=MIu^{(n)}-I(\phi^{(n)} v^{(n)}),\\
  & iD_-(Iv^{(n+1)}) =MIv^{(n)}-I(\phi^{(n)} u^{(n)}),\\
&Iu^{(n+1)}(0) = Iu_0\in L^2 , \quad Iv^{(n+1)}(0)=Iv_0\in L^2,
\end{aligned}
\right.
\end{equation}
 where    $$
  \square\phi^{(n)}=m^2\phi^{(n)} -  2\re(u^{(n)} \overline v^{(n)}),
$$
with the same data as for $\phi$.

Then, defining $y_n=\norm{Iw^n}_{X^{0,b}}$ for $n\ge 0$,
\eqref{est:Iw-reduc} becomes
$$
y_{n+1}\le CAN^{-s} + CBN^{-2s}\Gamma_1 y_n + C\Gamma_1\Gamma_2
y_n^3.
$$
By \eqref{Apriori-Dirac} and \eqref{u0v0p0}, $y_0\le 2CAN^{-s}$.
Now, if \eqref{T-reducedcond} holds, we conclude by induction that
$y_n\le 2CAN^{-s}$ for all $n\ge 0$. On the other hand, we know from
~\cite{st2007} that $(u^{(n)},v^{(n)})\rightarrow (u, v)\in
X_+^{s,b}\times X_-^{s,b}$ as $n\rightarrow \infty$, which implies
$Iw^{(n)}\rightarrow Iw\in X^{0,b}$ as $n\rightarrow \infty$, and
hence \eqref{Iw-bdd} follows.
\end{remark}
 It remain to prove the claim; i.e.,
\eqref{Ibilinear:phi-u}--\eqref{Ibilinear:u-v}. The estimates
\eqref{Ibilinear:phi-u} and \eqref{Ibilinear:phi-v} are symmetrical.
Hence we only prove \eqref{Ibilinear:phi-u} and
\eqref{Ibilinear:u-v}. As in Section \ref{sec-error-estimate}, we
decompose $u, v, \phi$ into low and high frequencies, and prove the
bilinear estimates for all possible interactions.
\subsection{Proof of \eqref{Ibilinear:phi-u}}
We recall that $s>-1/6, \ -s\le r<1/2+2s$, \ $b=1/2+\varepsilon$,
and the operator $I$ is the identity for low frequencies. Note also
that low-low interaction yields low frequency output. Then for the
low/low interaction, we have by \eqref{Ibdd-1} and the product
estimate \eqref{prodembed:-+}
\begin{equation}\label{uphi-ll1}
\begin{split}
\norm{I(\phi_l u_l)}_{X_-^{0, b-1}(S_{\Delta T})}&\le C \norm{\phi_l
u_l}_{X_-^{0, b-1}(S_{\Delta T})}\\
&\le C \norm{\phi_l}_{H^{2\varepsilon, b}(S_{\Delta
T})}\norm{u_l}_{X_+^{0, b}(S_{\Delta T})}.
\end{split}
\end{equation}
On the other hand, by \eqref{Ibdd-1}, Lemma \ref{Sob-emb-t},
H\"{o}lder in $t$,
 \eqref{Sob-prodembed:1} and \eqref{embxhc}, we
have
\begin{equation}\label{uphi-ll2}
\begin{split}
\norm{I(\phi_l u_l)}_{X_-^{0, b-1}(S_{\Delta T})}&\le C \norm{\phi_l
u_l}_{X_-^{0, b-1}(S_{\Delta T})}\\
&\le C\norm{\phi_l
u_l}_{L_t^{\frac{1}{1-2\varepsilon}}L_x^2(S_{\Delta
T})}\\
&\le C(\Delta T)^{1-2\varepsilon}\norm{\phi_l u_l}_{L_t^\infty
L^2_x(S_{\Delta
T})}\\
&\le C(\Delta T)^{1-2\varepsilon}\norm{\phi_l}_{L_t^\infty
H^{1/2+\varepsilon}_x(S_{\Delta T})}\norm{ u_l}_{L_t^\infty
L^2_x(S_{\Delta T})}
\\
 &\le C(\Delta T)^{1-2\varepsilon} \norm{\phi_l}_{H^{1/2+\varepsilon,b}(S_{\Delta
T})}\norm{u_l}_{X_+^{0, b}(S_{\Delta T})}.
\end{split}
\end{equation}
 Then interpolation between \eqref{uphi-ll1} and \eqref{uphi-ll2}, for $0\le \theta \le 1$, gives
\begin{align*}
&\norm{I(\phi_l u_l)}_{X_-^{0, b-1}(S_{\Delta T})}\\
 &\le C(\Delta T)^{(1-2\varepsilon)\theta} \norm{\phi_l}_{H^{2\varepsilon (1-\theta)+ (1/2+\varepsilon)\theta,b}(S_{\Delta
T})}\norm{u_l}_{X_+^{0, b}(S_{\Delta T})}.
\end{align*}
 We take
$\theta=\frac{2r-4s-4\varepsilon}{1-2\varepsilon}$ (by the
hypothesis made on $s ,r$, we then have $\theta\in [0, 1]$). This
implies $2\varepsilon (1-\theta)+ (1/2+\varepsilon)\theta=r-2s$ and
$(1-2\varepsilon)\theta=2r-4s-4\varepsilon$. Consequently,
 \begin{equation}\label{phi-u:ll}
\norm{I(\phi_l u_l)}_{X_-^{0, b-1}(S_{\Delta T})}
 \le C(\Delta T)^{2r-4s-4\varepsilon} \norm{\phi_l}_{H^{r-2s,b}(S_{\Delta
T})}\norm{u_l}_{X_+^{0, b}(S_{\Delta T})}.
\end{equation}

 The contribution from
low/high can be estimated using \eqref{Ibdd-1} and the product
estimate \eqref{prodembed:-+} as
\begin{equation}\label{phi-u:lh}
\begin{split}
\norm{I(\phi_l u_h)}_{X_-^{0, b-1}(S_{\Delta T})}&\le C \norm{\phi_l
u_h}_{X_-^{0, b-1}(S_{\Delta T})}\\
&\le C \norm{\phi_l}_{H^{r-2s,
b}(S_{\Delta T})}\norm{u_h}_{X_+^{2s-r+2\varepsilon, b}(S_{\Delta T})}\\
&= C\norm{\phi_l}_{H^{r-2s, b}(S_{\Delta
T})}N^{s} N^{-s}\norm{u_h}_{X_+^{s+s-r+2\varepsilon,b}(S_{\Delta T})}\\
&\le CN^{-r+2s+2\varepsilon} \norm{\phi_l}_{H^{r-2s, b}(S_{\Delta
T})}\norm{Iu_h}_{X_+^{0, b}(S_{\Delta T})}.
\end{split}
\end{equation}
The contribution from high/low can be estimated using
\eqref{Ibdd-1}, the product estimate \eqref{prodembed:-+}, and
\eqref{compare-Iz:z} as
\begin{equation}\label{phi-u:hl}
\begin{split}
\norm{I(\phi_h u_l)}_{X_-^{0, b-1}(S_{\Delta T})}&\le C \norm{\phi_h
u_l}_{X_-^{0, b-1}(S_{\Delta T})}\\
&\le C \norm{\phi_h}_{H^{2\varepsilon,
b}(S_{\Delta T})}\norm{u_l}_{X_+^{0, b}(S_{\Delta T})}\\
&\lesssim C N^{-r+s-s+2\varepsilon}\norm{\phi_h}_{H^{r-s+s,
b}(S_{\Delta
T})}\norm{u_l}_{X_+^{0,b}(S_{\Delta T})}\\
&= CN^{-r+s+2\varepsilon} \norm{I\phi_h}_{H^{r-s, b}(S_{\Delta
T})}\norm{u_l}_{X_+^{0, b}(S_{\Delta T})}\\
&\le CN^{-r+2s+2\varepsilon} \norm{I^2\phi_h}_{H^{r-2s, b}(S_{\Delta
T})}\norm{u_l}_{X_+^{0, b}(S_{\Delta T})}.\\
\end{split}
\end{equation}
 Similarly, we estimate the high/high interaction
using \eqref{Ibdd-1}, the product estimate \eqref{prodembed:-+}, and
\eqref{compare-Iz:z} as
\begin{equation}\label{phi-u:hh}
\begin{split}
\norm{I(\phi_h u_h)}_{X_-^{0, b-1}(S_{\Delta T})}&\le C \norm{\phi_h
u_l}_{X_-^{0, b-1}(S_{\Delta T})}\\
&\le C \norm{\phi_h}_{H^{-s+2\varepsilon,
b}(S_{\Delta T})}\norm{u_h}_{X_+^{s, b}(S_{\Delta T})}\\
&= C\norm{\phi_h}_{H^{r-s-r+2\varepsilon, b}(S_{\Delta T})}N^{s}
N^{-s}\norm{u_h}_{X_+^{s,b}(S_{\Delta T})}\\
&\le C N^{-r+s+2\varepsilon} \norm{I\phi_h}_{H^{r-s, b}(S_{\Delta
T})}\norm{Iu_h}_{X_+^{0, b}(S_{\Delta T})}\\
&\le C N^{-r+2s+2\varepsilon} \norm{I^2\phi_h}_{H^{r-2s,
b}(S_{\Delta T})}\norm{Iu_h}_{X_+^{0, b}(S_{\Delta T})}.
\end{split}
\end{equation}
Therefore, \eqref{Ibilinear:phi-u} follows from the estimates
\eqref{phi-u:ll}--\eqref{phi-u:hh}.
\subsection{Proof of \eqref{Ibilinear:u-v}}
We recall that  $s>-1/6, \ -s\le r<1/2+2s$ and $b=1/2+\varepsilon$.
Noting that $I$ is the identity for low frequencies, we have by
\eqref{Ibdd-1}, Lemma \ref{Thmembedding} and \eqref{Tfactor}
\begin{equation}\label{uv:ll}
\begin{split}
\norm{I^2(u_l \overline{v_l})}_{H^{r-2s-1, b-1}(S_{\Delta T})}&\le C\norm{u_l \overline{v_l}}_{H^{r-2s-1, -1/2+\varepsilon}(S_{\Delta T})}\\
&\le C \norm{u_l}_{X_+^{0,\varepsilon}(S_{\Delta
T})}\norm{v_l}_{X_-^{0,
\varepsilon}(S_{\Delta T})}\\
&\le  C(\Delta T)^{1-4\varepsilon}
\norm{u_l}_{X_+^{0,1/2-\varepsilon}(S_{\Delta
T})}\norm{v_l}_{X_-^{0,
1/2-\varepsilon}(S_{\Delta T})}\\
&\le C(\Delta T)^{1-4\varepsilon} \norm{u_l}_{X_+^{0,b}(S_{\Delta
T})}\norm{v_l}_{X_-^{0, b}(S_{\Delta T})}.
\end{split}
\end{equation}
The contribution from low/high interaction is estimated using
\eqref{Ibdd-1} and the product estimate \eqref{reduc-embedding:1} as
\begin{equation}\label{uv:lh}
\begin{split}
\norm{I^2(u_l \overline{v_h})}_{H^{r-2s-1, b-1}(S_{\Delta T})}&\le C\norm{u_l \overline{v_h}}_{H^{-1/2, b-1}(S_{\Delta T})}\\
&\le C \norm{u_l}_{X_+^{0, b}(S_{\Delta
T})}\norm{v_h}_{X_-^{-1/2+2\varepsilon,
b}(S_{\Delta T})}\\
&= C\norm{u_l}_{X_+^{0,b}(S_{\Delta
T})}\norm{v_h}_{X_-^{-1/2-s+2\varepsilon+s,
b}(S_{\Delta T})}\\
&\le C \norm{u_l}_{X_+^{0, b}(S_{\Delta T})}
N^{-1/2-s+2\varepsilon}\norm{v_h}_{X_-^{s,
b}(S_{\Delta T})}\\
&= CN^{-1/2+2\varepsilon}\norm{u_l}_{X_+^{0, b}(S_{\Delta T})}
\norm{Iv_h}_{X_-^{0, b}(S_{\Delta T})}.
\end{split}
\end{equation}
By symmetry
\begin{equation}\label{uv:hl}
\norm{I^2(u_h \overline{v_l})}_{H^{r-2s-1, b-1}(S_{\Delta T})}\le
CN^{-1/2+2\varepsilon}\norm{Iu_h}_{X_+^{0, b}(S_{\Delta T})}
\norm{v_l}_{X_-^{0, b}(S_{\Delta T})}.
\end{equation}
Finally, for the high/high interaction we obtain using
\eqref{Ibdd-1} and the product estimate
  \eqref{reduc-embedding:1}
 \begin{equation}\label{uv:hh}
\begin{split}
\norm{I^2(u_h \overline {v_h})}_{H^{r-2s-1, b-1}(S_{\Delta T})}&\le C\norm{u_h \overline{v_h}}_{H^{-1/2, b-1}(S_{\Delta T})}\\
&\le C \norm{u_h}_{X_+^{-1/4+2\varepsilon,
b}(S_{\Delta T})}\norm{v_h}_{X_-^{-1/4, b}(S_{\Delta T})}\\
&\le C N^{-1/4-s+2\varepsilon}\norm{u_h}_{X_+^{s,
b}(S_{\Delta T})}N^{-1/4-s}\norm{v_h}_{X_-^{s, b}(S_{\Delta T})}\\
& =CN^{-1/2+2\varepsilon} \norm{Iu_h}_{X_+^{0, b}(S_{\Delta
T})}\norm{Iv_h}_{X_-^{0, b}(S_{\Delta T})}.
\end{split}
\end{equation}
We therefore conclude that \eqref{Ibilinear:u-v} follows from the
estimates \eqref{uv:ll}--\eqref{uv:hh}.

\subsection*{Acknowledgements}
The author would like to thank Sigmund Selberg for continuous
support, encouragement and advice while writing this paper. The
author also thanks Damiano Foschi for helpful discussions and
generosity while I visited the University of Ferarra.

\end{document}